\documentclass{article}
\usepackage{amsthm}
\usepackage{amsfonts,amssymb,amsmath}
\usepackage{enumitem}
\usepackage{titlesec}
\usepackage{fouriernc}
\usepackage[T1]{fontenc}

\newlist{gcases}{enumerate}{1}
\setlist[gcases,1]{
  label={{\it Case}~{\it \Alph*}.},
  topsep=0ex,
  leftmargin=0in,
  labelsep=.1in,
  itemindent=.7in,
  itemsep=0ex
}
\newlist{tenumerate}{enumerate}{1}
\setlist[tenumerate,1]{
  label={(\arabic*)},
  topsep=0ex,
  leftmargin=.3in,
  labelsep=.1in,
  itemindent=0in,
  itemsep=0ex
}

\titleformat{\section}
  {\normalfont\bfseries}   
  {}                             
  {0pt}                          
  {Section \thesection\quad}    

\titleformat{name=\section,numberless}
  {\normalfont\bfseries}
  {}
  {0pt}
  {}

\newlength{\tabwidth}
\newlength{\tabheight}
\newlength{\tabrule}
\newlength{\tabwidthx}
\newlength{\tabheightx}

\def\gentabbox#1#2#3#4{\vbox to \tabheight{\setlength{\tabrule}{#3}%
  \setlength{\tabwidthx}{#1\tabwidth}\addtolength{\tabwidthx}{\tabrule}%

\setlength{\tabheightx}{#2\tabheight}\addtolength{\tabheightx}{-\tabheight}%
  \hbox to #1\tabwidth{%
 \hspace{-0.5\tabrule}\rule{\tabrule}{#2\tabheight}\hspace{-\tabrule}%
    \vbox to #2\tabheight{\hsize=\tabwidthx%
      \vspace{-0.5\tabrule}\hrule width\tabwidthx height\tabrule%
      \vspace{-0.5\tabrule}\vfil%
      \hbox to \tabwidthx{\hss#4\hss}%
        \vfil\vspace{-0.5\tabrule}%
      \hrule width\tabwidthx height\tabrule\vspace{-0.5\tabrule}}%
 \hspace{-\tabrule}\rule{\tabrule}{#2\tabheight}\hspace{-0.5\tabrule}}%
  \vspace{-\tabheightx}}}
\def\genblankbox#1#2{\vbox to \tabheight{\vfil\hbox to
#1\tabwidth{\hfil}}}
\def\tabbox#1#2#3{\gentabbox{#1}{#2}{0.4pt}{\strut #3}}

\catcode`\:=13 \catcode`\.=13 \catcode`\;=13
\catcode`\>=13 \catcode`\^=13
\def:#1\\{\hbox{$#1$}}
\def.#1{\tabbox{1}{1}{$#1$}}
\def>#1{\tabbox{2}{1}{$#1$}}
\def^#1{\tabbox{1}{2}{$#1$}}
\def;{\genblankbox{1}{1}\relax}
\catcode`\:=12 \catcode`\.=12 \catcode`\;=12
\catcode`\>=12 \catcode`\^=7

\newenvironment{tableau}{\bgroup\catcode`\:=13 \catcode`\.=13
  \catcode`\;=13 \catcode`\>=13 \catcode`\^=13
  \setlength{\tabheight}{3ex}\setlength{\tabwidth}{3ex}%
  \def\b##1##2##3{\gentabbox{##1}{##2}{1.2pt}{\vbox{##3}}}%
  \def\n##1##2##3{\gentabbox{##1}{##2}{0.4pt}{\vbox{##3}}}%
  \vbox\bgroup\offinterlineskip}{\egroup\egroup}
\newtheoremstyle{mytheoremstyle} 
    {\topsep}                    
    {\topsep}                    
    {}                   
    {}                           
    {}                   
    {.}                          
    {0.5em}                       
    {\thmnumber{4.#2.} \thmname{#1}}  

\newtheoremstyle{myremarkstyle} 
    {\topsep}                    
    {\topsep}                    
    {}                   
    {}                           
    {}                   
    {.}                          
    {0.5em}                       
    {\thmname{#1}}  

\theoremstyle{mytheoremstyle}
\newtheorem{theorem}{THEOREM}[section]
\newtheorem{corollary}[theorem]{COROLLARY}
\newtheorem{lemma}[theorem]{LEMMA}
\newtheorem{proposition}[theorem]{PROPOSITION}
\newtheorem{definition}[theorem]{DEFINITION}
\newtheorem{notation}[theorem]{NOTATION}
\newtheorem{remark}[theorem]{REMARK}

\theoremstyle{myremarkstyle}
\newtheorem{remark*}{REMARK}

\newcommand\T{\mathbf{T}}
\newcommand\Ss{\mathbf{S}}
\newcommand\oT{\overline{\mathbf{T}}}
\newcommand\tT{\tilde{\mathbf{T}}}
\newcommand\LL{\mathbf{L}}
\newcommand\R{\mathbf{R}}
\newcommand\EE{\mathbf{E}}
\newcommand\SC{\textup{SC}}
\newcommand\In{\textup{In}}
\newcommand\OC{\textup{OC}}
\newcommand\RE{\textup{Re}}
\newcommand\sh{\textup{Shape}}
\newcommand\E{\mathbf{E}}
\newcommand\ab{{\alpha \beta}}
\newcommand\ba{{\beta \alpha}}
\newcommand\Ab{\mathcal{A}_\beta}
\newcommand\Bb{\mathcal{B}_\beta}

\begin{document}

\noindent
{\bf \large On the classification of primitive ideals for complex \\ classical Lie algebras, IV}

\vspace{.3in}
\noindent
THOMAS PIETRAHO \\
\vspace{.1in}
{\it \small Department of Mathematics, Bowdoin College, Brunswick, ME 04011, USA}

\noindent
W.~M.~MCGOVERN \\
\vspace{.1in}
{\it \small Department of Mathematics, University of Washington, Seattle, WA, 98195, USA}

\noindent
DEVRA GARFINKLE

\section*{Introduction}
This paper is the fourth and last in the series \lq\lq On the classification of primitive ideals for complex classical Lie algebras", extending the main results of \cite{garfinkle3} to type $D$.  Here the generalized $\tau$-invariant defined in \cite{garfinkle3} must be defined in a different way, as there is no pair of adjacent simple roots in type $D$ with different lengths.  Instead we attach a family of operators to a quadruple of simple roots spanning a subsystem of type $D_4$, each of them taking either one or two values, like the corresponding operators in types $B$ and $C$, and use these to characterize primitive ideals via their generalized $\tau$-invariants.

\section{}

\begin{notation}

We will implicitly rely on the notation introduced in the first three parts of this sequence of papers.  It is summarized in an appendix of \cite{garfinkle3}.

 \begin{tenumerate}
\item For simple roots we take $\alpha_1' = e_2+e_1$ and $\alpha_m = e_m - e_{m-1}$ for $2 \leq m \leq n$. This non-standard notation will facilitate the use of results from \cite{garfinkle3}.  Set $\Pi' = \{ \alpha_1', \alpha_2, \ldots, \alpha_n\}$ and write $s_1'$ for $s_{\alpha_1'}$, the reflection in $\alpha_1'$. We let $W$ be the Weyl group of type $C_n$ and will write $W'$ for the subgroup of $W$ generated by the reflections in the roots belonging to $\Pi'$.  Further, let $W''=W\setminus W'$.

\item We define $$\mathcal{S}'(M_1,M_2) = \{\gamma \in \mathcal{S}(M_1,M_2) \, | \, |\{(e,g,\epsilon) \in \gamma \, | \epsilon = -1\}| \equiv 0  \pmod{2} \},$$  and set $\mathcal{S}''(M_1,M_2) = \mathcal{S}(M_1,M_2) \setminus \mathcal{S}'(M_1,M_2)$.  Define $\mathcal{S}'(n,n)$ and $\mathcal{S}''(n,n)$ by imitating \cite[1.1.2]{garfinkle1}.

\item For $(\T_1,\T_2) \in \mathcal{T}(M_1,M_2)$ we write $n_h((\T_1,\T_2)) = n_h(\T_1)+n_h(\T_2)$, and similarly define $n_v((\T_1,\T_2)$.

\item We define $$\mathcal{T}'(M_1,M_2) = \{(\T_1,\T_2) \in \mathcal{T}(M_1,M_2) \, | \, n_v((\T_1,\T_2)) \equiv 0 \pmod{4}\}$$ and $$\mathcal{T}''(M_1,M_2) = \{(\T_1,\T_2) \in \mathcal{T}(M_1,M_2) \, | \, n_v((\T_1,\T_2)) \equiv 2 \pmod{4}\}.$$  The sets $\mathcal{T'}(n,n)$ and $\mathcal{T}''(n,n)$ are defined by imitating \cite[1.1.9]{garfinkle1}.  When we wish to specify that the $D$ grid is being used explicitly, we will write $\mathcal{T}'_D(M_1,M_2)$, $ \mathcal{T}''_D(M_1,M_2)$, $\mathcal{T}'_D(n,n),$ and $\mathcal{T}''_D(n,n)$.  Write $\mathcal{T}_D(n,n)$ for the union of $\mathcal{T}'_D(n,n)$ and $\mathcal{T}''_D(n,n)$.
\end{tenumerate}

\end{notation}

\begin{remark}
    For $w \in W$ we have $\delta(s_1'w) = \SC^R(1;\SC^R(2;\In^R(1,2;\delta(w))))$ and $\delta(ws_1') = \SC^L(1;\SC^L(2;\In^L(1,2;\delta(w))))$. Consequently, we have $$\delta(W') = \mathcal{S}'(n,n) \text { and } \delta(W'')= \mathcal{S}''(n,n).$$
\end{remark}

\begin{proposition}\label{proposition:nhnv}
Let $(\T',v,\epsilon) \in M$ and let $(\T,P) = \alpha((\T',v,\epsilon))$.

\begin{tenumerate}
    \item If $\epsilon =1$, then either
    \begin{enumerate}[label=(\alph*), leftmargin=.6in, topsep=0ex, itemsep=0ex]
        \item $P$ is horizontal, $n_h(\T)=n_h(\T')+1$, and $n_v(\T) = n_v(\T')$, or
        \item $P$ is vertical, $n_h(\T) = n_h(\T') +2$, and $n_v(\T) = n_v(\T')-1$.
    \end{enumerate}
    \item If $\epsilon =-1$, then either
    \begin{enumerate}[label=(\alph*), leftmargin=.6in, topsep=0ex, itemsep=0ex]
        \item $P$ is horizontal, $n_h(\T)=n_h(\T')-1$, and $n_v(\T) = n_v(\T')+1$, or
        \item $P$ is vertical, $n_h(\T) = n_h(\T')$, and $n_v(\T) = n_v(\T')+2$.
    \end{enumerate}
\end{tenumerate}
\end{proposition}

\begin{proof}
The proof is by induction on $|M|$; the case when $|M|=1$ is obvious.
Set $e = \sup M$ and note that the proof of the proposition is clear when $v=e$, so assume $v \neq e$.  Set $\overline{\T}' = \T' -e$ and $(\overline{\T},\overline{P}) = \alpha((\overline{\T}',v,\epsilon)).$  Then $\overline{\T}=\T - e.$  Set $P_e = P(e,\T)$ and $P_e' = P(e,\T')$.  We will prove statement (1) for $n_h$, the others can be verified by a similar case-by-case analysis.  There are four cases.

\begin{gcases}
 \item Here $\overline{P}$ is horizontal and $|\overline{P}\cap {P}_e'|=1$.  Then $P_e'$ and $P$ are vertical and $P_e$ is horizontal.  Thus $n_h(\T) = n_h(\overline{\T})+1 = n_h(\overline{\T}')+2=n_h(\T')+2$, as desired.

 \item Here $\overline{P}$ is horizontal and $\overline{P} =\overline{P}_e'$.  Then $P_e'$, $P$, and $P_e$ are horizontal.  Thus $n_h(\T) = n_h(\overline{\T})+1 = n_h(\overline{\T}')+2=n_h(\T')+2$.

 \item Here $\overline{P}$ is vertical and $|\overline{P}\cap {P}_e'|=1$.  Then $P_e'$ and $P$ are horizontal and $P_e$ is vertical.  Thus $n_h(\T) = n_h(\overline{\T}) = n_h(\overline{\T}')+2=n_h(\T')+1$.

 \item Here $\overline{P}$ is vertical and $\overline{P} =\overline{P}_e'$.  Then $P_e'$, $P$, and $P_e$ are vertical.  Thus $n_h(\T) = n_h(\overline{\T}) = n_h(\overline{\T}')+2=n_h(\T')+2$.
\end{gcases}
\end{proof}

\begin{corollary}
Let $\gamma \in \mathcal{S}(M_1,M_2)$.  Then $n_h(A(\gamma)) \equiv 0 \pmod{2} $ and $n_v(A(\gamma)) \equiv 0 \pmod{2}.$  We have $\gamma \in \mathcal{S}'(M_1,M_2)$ if and only if $n_v(A(\gamma)) \equiv 0 \pmod{4}$ and $\gamma \in \mathcal{S}''(M_1,M_2)$ if and only if $n_v(A(\gamma)) \equiv 2 \pmod{4}.$
\end{corollary}

\begin{proof}
The proof is by induction on $|M_1|$, and is then obvious from the definitions and Proposition 4.\ref{proposition:nhnv}.
\end{proof}

\begin{corollary}
    \begin{tenumerate}
        \item We have $\mathcal{T}(M_1,M_2) = \mathcal{T}'(M_1,M_2) \cup \mathcal{T}''(M_1,M_2)$.
        \item The map $A$ yields a bijection from $\mathcal{S}'(M_1,M_2)$ to $\mathcal{T}'(M_1,M_2)$ and from $\mathcal{S}''(M_1,M_2)$ to $\mathcal{T}''(M_1,M_2)$.
    \end{tenumerate}
\end{corollary}

\begin{corollary}
The map $A \circ \delta$ yields a bijection from $W'$ to $\mathcal{T}'(n,n)$ and from $W''$ to $\mathcal{T}''(n,n)$.
\end{corollary}

\section{}

\begin{notation}\label{notation:section2}  Let $(\T_1,\T_2) \in \mathcal{T}_K(M_1,M_2)$ for $K = B,$ $C$, or $D$.  Let $c$ be an extended cycle in $\T_1$ relative to $\T_2$ and suppose that $c$ is a union of open cycles in $\T_1$.  We write
$$Cy(c) = \{ c' \; | \; c' \subseteq c \text{ and } c' \in \OC(\T_1)\},$$
and
$$Sq(c) = \{S \; | \; \text{there is a } c' \in Cy(c) \text{ such that } S=S_b(c') \text{ or } S=S_f(c')\}.$$  We let $n_u(c)$ be the number of up cycles contained in $c$ and $n_d(c)$ be the number of down cycles contained in $c$.

\end{notation}

\begin{proposition}\label{proposition:nuextendedcycle} Let $(\T_1,\T_2)$ and $c$ be as in 4.\ref{notation:section2} and suppose that $Cy(c) \subseteq \OC^*(\T_1)$.  Let $\overline{c}$ be the extended cycle in $\T_2$ relative to $\T_1$ corresponding to $c$.  Then either
$n_u(c)+n_u(\overline{c}) = n_d(c)+n_d(\overline{c})+2$ or
$n_u(c)+n_u(\overline{c}) = n_d(c)+n_d(\overline{c})-2$.

\end{proposition}

\begin{proof} The proof is by induction on $|M_1|$, the proposition being vacuously true when $|M_1|=0.$ Let $e = \sup M_1$.  If $e \notin c$ then the proposition in true by induction using Theorem 2.2.3 and Proposition 2.2.4, cf. Proposition 2.3.3(b)).  So assume $e\in c$.  The proposition is clear if $c = \{e\}$ so assume $c \neq \{e\}$.  Let $(\T_1',\T_2')=(\T_1,\T_2) - L$.  Then by Proposition 2.3.3(b), $c \setminus \{e\}$ is an extended cycle in $\T_1'$ relative to $\T_2'$.  Set $c' = c \setminus \{e\}$ and let $\overline{c}'$ be the corresponding extended cycle in $\T_2'$ relative to $\T_1'$.  We will assume that $P(e, \T_1)$ is horizontal; an analogous transposed argument works when $P(e, \T_1)$ is vertical.  If $P'(e,T_1)$ is also horizontal, that is, if for some $\tilde{c} \in Cy(c)$, $\tilde{c} \neq \{e\}$, we have $S_b(\tilde{c}) \in P(e,\T_1)$ or $S_f(\tilde{c}) \in P'(e, \T_1)$, then the proposition is clear using the c.s.p.b's of Theorem 2.2.3 (1a) or (2a) and Proposition 2.2.4 part (1a) or (2a).  So assume $P'(e,\T_1)$ is vertical.  Set $P(e, \T_1) = \{ S_{ij}, S_{i,j+1}\}.$  There are two cases:
\begin{gcases}
    \item Here $\{e\} \in Cy(c)$.  Then $P'(e,\T_1) = \{S_{ij}, S_{i+1,j}\}.$  Thus $\{e\}$ is a down cycle, and so by using Proposition 2.2.4 (1b) we have $n_u(c)=n_u(c')$ and $n_d(c) = n_d(c')+1.$  Now by Theorem 2.2.3 (1b) there are cycles $c^1$ and $c^2 \in \OC^*(\T_2)$ and $\tilde{c}' \in \OC^*(\T_2')$ such that $S_b(c^1)=S_b(\tilde{c}')$, $S_f(c^1)=S_{i+1,j}$, $S_b(c^2) = S_{i,j+1}$, and $S_f(c^2) = S_f(\tilde{c}').$  Note that the alternative suggested by Theorem 2.2.3 (1b) is the existence of a cycle $\tilde{c} \in \OC^*(\T_2)$ with $S_b(\tilde{c}) = S_{i+1,j}$ and $S_f(\tilde{c})= S_{i,j+1}$, but then $\{e\}$  is an extended cycle in $\T_1$ relative to $\T_2$, contradicting the assumption that $c \neq \{e\}$.  Let $U = \mu(Cy(\overline{c}) \setminus \{c^1,c^2\})$ where $\mu$ is the c.s.p.b. of Theorem 2.2.3 (1b).  Then $\overline{c}' = (\bigcup_{\tilde{c} \in U} \tilde{c}) \cup \tilde{c}'.$  There are three possibilities.  Suppose first that $c^1$ is a down cycle.  Then $c^2$ is nested in $c^1$ and thus we have $\rho(S_b(\tilde{c}'))=\rho(S_b(c^1))<\rho(S_f(c^2))<i$, so $c^2$ is an up cycle and $\tilde{c}'$ is a down cycle.  Similarly, if $c^2$ is a down cycle, then $c^1$ is an up cycle and $\tilde{c}'$ is a down cycle.  Finally, if both $c^1$ and $c^2$ are up cycles then $\tilde{c}'$ is an up cycle.  In all three cases we have $n_u(\overline{c})= n_u(\overline{c}')+1$ and $n_d(\overline{c}) = n_d(\overline{c}')$.  Thus
        $$n_u(c)+n_u(\overline{c}) = n_u(c')+n_u(\overline{c}')+1 \text{ and } n_d(c)+n_d(\overline{c}) = n_d(c')+n_d(\overline{c}')+1,$$ so the  proposition holds by induction.
    \item Here $\{e\}$ is not a cycle in $\T_1$.  Then $P'(e,\T_1) =\{S_{i-1,j+1},S_{i,j+1}\}$.  By Theorem 2.2.3 (2b), there is a cycle $\tilde{c}' \in \OC^*(\T_2')$ such that $S_b(\tilde{c}') = S_{i-1,j+1}$ and $S_f(\tilde{c}')=S_{ij}$.  In particular $\tilde{c}'$ is a down cycle so the c.s.p.b. of Theorem 2.2.3 (2b) implies $n_u(\overline{c}) = n_u(\overline{c}')$ and $n_d(\overline{c})=n_d(\overline{c}')-1$.  By Proposition 2.2.4 (2b), there are cycles $\tilde{c} \in \OC^*(\T_1)$ and $c^1$, $c^2\in \OC^*(\T_1')$ such that $S_b(c^1)=S_b(\tilde{c})$, $S_f(c^1)=S_{ij}$, $S_b(c^2) = S_{i-1,j+1}$, and $S_f(c^2)=S_f(\tilde{c})$.  We now proceed as in the previous case, finding $n_u(c)=n_u(c')-1$ and $n_d(c)=n_d(c')$.  Thus again the proposition is true by induction.
\end{gcases}

\end{proof}

\begin{definition}\label{definition:upextendedcycle}
Let $(\T_1,\T_2) \in \mathcal{T}_K(M_1,M_2)$ for $K = B$, $C$, or $D$.  Let $c$ be an extended cycle in $\T_1$ relative to $\T_2$ and suppose either $c$ is an closed cycle in $\T_1$ or $Cy(c) \subseteq \OC^*(\T_1)$.  We say that $c$ is an up extended cycle is either $c$ is a closed up cycle in $\T_1$ or $n_u(c)+n_u(\overline{c}) = n_d(c) + n_d(\overline{c}) +2$.  Otherwise, we say that $c$ is a down extended cycle.
\end{definition}

\begin{proposition} \label{proposition:nhmt}
Let $(\T_1,\T_2)$ and $c$ be as in Definition 4.\ref{definition:upextendedcycle}.  Set $(\T_1',\T_2') = \E((\T_1,\T_2),c,L).$
    \begin{tenumerate}
        \item If $c$ is an up extended cycle, then $$n_h((\T_1',\T_2'))=n_h((\T_1,\T_2))+2 \text{ and } n_v((\T_1',\T_2'))= n_v((\T_1,\T_2))-2.$$
        \item If $c$ is a down extended cycle, then $$n_v((\T_1',\T_2'))=n_v((\T_1,\T_2))+2 \text{ and } n_h((\T_1',\T_2'))= n_h((\T_1,\T_2))-2.$$
    \end{tenumerate}
\end{proposition}
\begin{proof}
If $c$ is a closed cycle in $\T_1$, then this is given by conditions (1) and (2) of Propositions 3.3.10 and 3.3.11.  If $c$ is a union of open cycles in $\T_1$ with $Cy(c) \subseteq \OC^*(\T_1)$, then this proposition is a consequence of Proposition 4.\ref{proposition:nuextendedcycle} and conditions (3) and (4) of Propositions 3.3.7 and 3.3.8.
\end{proof}

\begin{remark} We have now proved the appropriate analogues of conditions (3) and (4) of Propositions 3.3.7 and 3.3.8 for up and down extended cycles $c$ with $Cy(c) \subseteq \OC^*(\T_1)$. The relevant condition (5) of both propositions is also true via induction.
Conditions (1) and (2) can be modified as below.  We state the result for up extended cycles; analogous statements hold for down extended cycles.
\end{remark}

\begin{proposition}
Let $(\T_1,\T_2) \in \mathcal{T}_K(M_1,M_2)$ for $K = B$, $C$, or $D$ and let $c$ be an up extended cycle in $\T_1$ relative to $\T_2$.  If $\hat{c} \in Cy(c)$, then either $\hat{c}$ is an up cycle or there is an up cycle $\hat{c}^1 \in Cy(c)$ with $\hat{c}$ nested in $\hat{c}^1$.  Let $k = |Cy(c)|$ and label the squares $Sq(c)=\{S_1, \ldots, S_{2k}\}$ so that $\rho(S_1) < \ldots < \rho(S_{2k}),$ or equivalently, $\kappa(S_1) > \ldots > \kappa(S_{2k})$.  Then
    \begin{tenumerate}
        \item  $S_1=S_f(c')$ for some $c' \in Cy(c)$ and $S_{2k}=S_b(c'')$ for some $c'' \in Cy(c)$, and
        \item $S_i = S_f(c')$ for some $c' \in Cy(c)$ if and only if $S_{i+1} = S_b(c'')$ for some $c'' \in Cy(c)$.
    \end{tenumerate}
\end{proposition}

\begin{proof}
The proof is by induction as for Proposition 4.\ref{proposition:nuextendedcycle}.  The only difficulty arises in  Case A, so let $e$, $(\T_1', \T_2')$, $c^1$, $c^2,$ and $c'$ be as therein.
Assume first that $c^1$ is an up cycle.  Note that there is an even number of squares $S \in Sq(c')$ with $i+1<\rho(S)<\rho(S_b(c^1))$, since any cycle in $\T_2$ with such an $S$ as its forward or back square is nested in $c^1$.
By (2) and induction, these alternate between back and forward squares, so if we pick $S \in Sq(c')$ with $\rho(S)$ maximal given that $\rho(S) \leq \rho(S_b(c^1))$, then there is a $\hat{c}^2 \in Cy(c')$ such that $S=S_b(\hat{c}^2).$
If $\hat{c}^2$ is an up cycle, the $\hat{c}=\{e\}$ is nested in $\hat{c}^2$, otherwise by induction $\hat{c}^2$ is nested in some up cycle in $c'$ which then also nests $\{e\}$.  We see that statement (2) is also satisfied by $c$.  The proof is similar when $c^2$ is an up cycle.  Finally, as in Proposition 4.\ref{proposition:nuextendedcycle}, we must also consider the proof of the transpose of Case A.  This is trivial for (1), and for (2), it is the transposed proof.

\end{proof}


\section{}

\begin{definition}
Let $w \in W'$ or $w\in W"$.  We define $\tau^L(w)$, $\tau^R(w) \subseteq \Pi'$ as in Definition 2.1.3 using $\alpha'$ and $s_1'$ is place of $\alpha_1$ and $s_1$.  We define $D_{\alpha \beta}^L(W')$ and $D_{\alpha \beta}^R(W')$ for any pair of adjacent simple roots $\alpha, \beta \in \Pi'$ as in Definition 2.1.4-1 and $T_{\alpha \beta}^L$ and $T_{\alpha \beta}^R$ as in Definition 2.1.4-2.

\end{definition}

\begin{definition} \label{definition:tau}Let $\gamma \in \mathcal{S}(M_1,M_2)$.
    \begin{tenumerate}
        \item Suppose $\{1,2\} \subseteq M_1$.  Let $k$, $l$, $\epsilon_1$, and $\epsilon_2$, be such that $\{(1,k,\epsilon_1),(2,l,\epsilon_2)\} \subseteq \gamma.$  We say that $\alpha_1' \in \tau^L(\gamma)$ if either both $k>l$ and $\epsilon_1=-1$ or both $k<l$ and $\epsilon_2=-1$;  otherwise we say $\alpha_1' \notin \tau^L(\gamma)$.  We define analogously $\alpha_1' \in \tau^R(\gamma)$ and $\alpha_1' \notin \tau^R(\gamma)$ when $\{1,2\} \subseteq M_2$.
        \item For $i \geq 2$, the statements $\alpha_i \in \tau^L(\gamma)$, $\alpha_i \notin \tau^L(\gamma)$, $\alpha_i \in \tau^R(\gamma)$, and $\alpha_i \notin \tau^R(\gamma)$ are defined as in Definition 2.1.5.
        \item  Suppose that either  $\{1,2,3\} \subseteq M_1$ and $\{\alpha, \beta\} = \{\alpha_1', \alpha_3\}$ or $\{i-1,i,i+1\} \subseteq M_1$ and $\{\alpha, \beta\}= \{ \alpha_i, \alpha_{i+1}\}$.  Then we define $D_{\alpha \beta}^L(\mathcal{S}(M_1,M_2))$ and $D_{\alpha \beta}^R(\mathcal{S}(M_1,M_2))$ as in Definition 2.1.7-1.
        \item If $\{\alpha, \beta\} = \{\alpha_1', \alpha_3\}$ and $\{1,2,3\} \subseteq M_1$ , we define $$T^L_{\alpha \beta}: D_{\alpha \beta}^L(\mathcal{S}(M_1,M_2)) \longrightarrow D_{\beta \alpha}^L(\mathcal{S}(M_1,M_2))$$ by setting $T_{\alpha \beta}(\gamma)= \gamma'$ where
            $$\{\gamma'\} = \{\SC^L(1;\SC^L(2;\In(1,2;\gamma))),\In^L(2,3;\gamma)\} \cap D_{\beta \alpha}^L(\mathcal{S}(M_1,M_2)).$$  As usual, $T_{\alpha \beta}^L$ is a well-defined bijection with inverse $T_{\beta \alpha}^L$.
        \item If $\{i-1,i,i+1\} \subseteq M_1$ and $\{\alpha, \beta\}= \{ \alpha_i, \alpha_{i+1}\}$ with $i \geq 2$, we define
            $$T^L_{\alpha \beta}: D_{\alpha \beta}^L(\mathcal{S}(M_1,M_2)) \longrightarrow D_{\beta \alpha}^L(\mathcal{S}(M_1,M_2))$$
            as in Definition 2.1.7-2.
        \item We define analogously the operators  $$T^L_{\alpha \beta}: D_{\alpha \beta}^R(\mathcal{S}(M_1,M_2)) \longrightarrow D_{\beta \alpha}^R(\mathcal{S}(M_1,M_2)).$$
    \end{tenumerate}
\end{definition}

\begin{remark*}  All of the following properties of the $\tau$-invariant also hold with $R$ in place of $L$:
    \begin{tenumerate}
        \item For $w \in W'$, we have $\tau^L(\delta(w))=\tau^R(w)$ and  $\tau^R(\delta(w))=\tau^L(w)$.  For $w\in  D_{\alpha \beta}^L(W')$ we have $\delta(T_{\alpha \beta}(w)) = T_{\alpha \beta}(\delta(w))$.
        \item If $\gamma \in \mathcal{S}(M_1,M_2)$ and $\{1,2\} \subseteq M_1$, then $\alpha_1' \in \tau^L(\gamma)$ if and only if $\alpha_1' \notin \tau^L(\SC(\gamma))$.
        \item If $\gamma \in D_{\alpha \beta}(\mathcal{S}(M_1,M_2))$, then $\SC(\gamma) \in D_{\beta \alpha}(\mathcal{S}(M_1,M_2))$ and $$T_{\alpha \beta}(\SC(\gamma)) = \SC(T_{\alpha \beta}(\gamma)).$$
    \end{tenumerate}
\end{remark*}

\begin{remark*} As noted in Remark 2.1.5, the notion of inclusion in a $\tau$-invariant for elements of $\mathcal{S}(M_1,M_2)$ is not defined as inclusion in a set, but rather as formal statement.  We remedy this as follows.  Let $M \subset \mathbb{N}^*$ be a finite set.  Define $\tau(M)$ as a subset of $\{\alpha_1'\} \cup \{\alpha_i \; | \; 1 < i < \infty \}$ by
    \begin{tenumerate}
        \item $\alpha_1' \in \tau (M)$ if and only if $\{1,2\} \subseteq M$, and
        \item for $i>1$, $\alpha_i \in \tau(M)$ if and only if $\{i,i+1\} \subseteq M$.
    \end{tenumerate}
We then define, say,  for $\gamma \in \mathcal{S}(M_1,M_2)$, $\tau^L(\gamma) \subseteq \tau(M_1)$, adapting the previous definitions to this context in the obvious way.  Strictly speaking, the $\tau$-invariants of this paper must consequently be distinguished from those
in \cite{garfinkle2} and \cite{garfinkle3}, but we will neglect doing so.
\end{remark*}

\begin{definition} \label{definition:alpha1tau} Let $\T \in \mathcal{T}(M)$ and suppose $\{1,2\} \subseteq M$.  We say that $\alpha_1' \in \tau(\T)$ if either:
    \begin{tenumerate}
        \item $P(1,\T) = \{S_{1,1},S_{2,1}\}$ and $P(2,\T) = \{S_{1,2},S_{2,2}\}$,
        \item $P(1,\T) = \{S_{1,1},S_{1,2}\}$ and $P(2,\T) = \{S_{2,1},S_{3,1}\}$, or
        \item $P(1,\T) = \{S_{1,1},S_{2,1}\}$ and $P(2,\T) = \{S_{3,1},S_{4,1}\}$.
    \end{tenumerate}
Otherwise we say that $\alpha_1' \notin \tau(\T)$.
\end{definition}

\begin{remark*}  With $\T$ as above, we have $\alpha_1' \in \tau(\T)$ if and only if $\alpha_1' \notin \tau(^t\T)$.
\end{remark*}

We will view the map $A$ of Definition 1.2.1 as a bijection from $\mathcal{S}(M_1,M_2)$ to $\mathcal{T}_D(M_1,M_2)$ and write $A(\gamma) = (\LL(\gamma), \R(\gamma))$.  The following is obvious by inspection:
\begin{proposition} \label{proposition:alpha1tauRSK} Suppose $\gamma \in \mathcal{S}(M_1,M_2)$ and $\{1,2\} \subseteq M_1.$  Then $\alpha_1' \in \tau^L(\gamma)$ if and only if $\alpha_1' \in \tau(\LL(\gamma))$ and $\alpha_1' \in \tau^R(\gamma)$ if and only if $\alpha_1' \in \tau(\R(\gamma))$
\end{proposition}

When $\{\alpha, \beta\} = \{\alpha_i, \alpha_{i+1}\}$ or $\{\alpha_1', \alpha_3\}$, we define the domains $D_{\alpha \beta}^L(\mathcal{T}_D(M_1,M_2))$ and $D_{\alpha \beta}^R(\mathcal{T}_D(M_1,M_2))$
as in Definition 2.1.11.  When $\{\alpha, \beta\} = \{\alpha_i, \alpha_{i+1}\}$, we also adopt its subsequent definitions of $T_{\alpha \beta}^L$ and $T_{\alpha \beta}^R.$  When $\{\alpha, \beta\} = \{\alpha_1', \alpha_3\}$, the definition is more intricate, as below:

\begin{definition} \label{definition:T13} Suppose $\{\alpha, \beta\} =\{\alpha_1', \alpha_3\}$.  We define
$$T_{\alpha \beta}^L: D_{\alpha \beta}^L (\mathcal{T}_D(M_1,M_2)) \longrightarrow D_{\beta \alpha}^L (\mathcal{T}_D(M_1,M_2))$$ via several cases.  Let $(\T_1, \T_2) \in D_{\alpha \beta}^L (\mathcal{T}_D(M_1,M_2)).$
    \begin{tenumerate}
        \item Suppose $S_{2,2} \notin P(2,\T_1) \cup P(3,\T_1)$.  Then we set $T_{\alpha \beta}^L((\T_1,\T_2))= (\In(2,3;\T_1),\T_2).$
        \item Suppose $P(1,\T_1) = \{S_{1,1},S_{2,1}\}$, $P(2,\T_1) = \{S_{1,2},S_{2,2}\}$, and $P(3,\T_1) = \{S_{1,3},S_{1,4}\}$.  Set $(\oT_1,\oT_2)=\EE((\T_1,\T_2), ec(3,T_1;T_2),L).$ Let $\T_1'=(\oT_1 \setminus \{(2,S_{2,2}), (3, S_{1,3})\}) \cup \{(2,S_{1,3}),(3,S_{2,2})\}.$ We then define $$T_{\alpha \beta}^L((\T_1,\T_2))= (\T_1',\oT_2).$$
        \item Suppose $P(1,\T_1) = \{S_{1,1},S_{2,1}\}$, $P(2,\T_1) = \{S_{1,2},S_{2,2}\}$, and $P(3,\T_1) = \{S_{1,3},S_{2,3}\}$.  Set $\T_1'= (\T_1 \setminus \{(2,S_{2,2}),(3,S_{1,3})\}) \cup\{(2,S_{1,3}),(3,S_{2,2})\}.$  We then define $$T_{\alpha \beta}^L((\T_1,\T_2))= \EE((\T_1',\T_2),ec(3,\T_1';T_2),L).$$
        \item In any of the settings not listed among (1)-(3), $T_{\alpha \beta}^L((\T_1,\T_2))$ is defined by the rules $T_{\alpha \beta}^L= (T_{\beta \alpha}^L)^{-1}$ and $^{t}(T_{\alpha \beta}^L((\T_1,\T_2))) = T_{\beta \alpha}^L(^{t}(\T_1,\T_2))$.
    \end{tenumerate}
The operators $T_{\alpha \beta}^R$ are defined an analogous manner.
\end{definition}

\begin{remark*}  An alternate description of $T_{\alpha \beta}((\T_1,\T_2))$ in case (2) can be given as follows.  Let
\begin{align*}
H_1 & = \{( 1, S_{1,1}), ( 1, S_{2,1}), ( 2, S_{1,2}),  (2, S_{2,2}), ( 3, S_{1,3}), ( 3, S_{1,4})\} \text{ \hspace{.1in} and } \\
 H_2 & = \{( 1, S_{1,1}), ( 1, S_{1,2}), ( 2, S_{1,3}),  ( 2, S_{1,4}), ( 3, S_{2,1}), ( 3, S_{2,2})\}.
\end{align*}
Set $\tT_1=(\T_1 \setminus H_1) \cup H_2.$ Then
$T_{\alpha \beta}^L((\T_1,\T_2)) = \EE((\tT_1,\T_2),ec(3,\tT_1;\T_2),L).$  Similarly, in case (3), we set $(\oT_1,\oT_2) = \EE((\T_1,\T_2), ec(3,\T_1;\T_2),L)$.  Then
$T_{\alpha \beta}^L((\T_1,\T_2)) = ((\oT_1 \setminus H_1)\cup H_2, \oT_2).$
\end{remark*}

\begin{proposition}  Let $(\T_1,\T_2) \in D_{\alpha \beta}^L(\mathcal{T}_D(M_1,M_2)).$  We have $$(\T_1,\T_2) \in \mathcal{T}'(M_1,M_2)\text{\hspace{.1in} if and only if \hspace{.1in} } T_\ab^L(\T_1,\T_2) \in \mathcal{T}'(M_1,M_2).$$  The corresponding statement holds when $L$ is replaced by $R$.
\end{proposition}

    \begin{proof}  If $\{\alpha, \beta\} = \{\alpha_i, \alpha_{i+1}\}$ for some $i \geq 2$, then this is clear from the definitions.  If on the other hand $\{\alpha, \beta\} = \{\alpha_1', \alpha_3\}$, we also can use Proposition 4.\ref{proposition:nhmt}.
    \end{proof}

\begin{lemma}\label{lemma:AofSC}  Let $\gamma \in \mathcal{S}(M_1,M_2)$ and suppose that $1 \in M_1$.  Let $\gamma'=\SC^L(1;\gamma)$.  Then
$$A(\gamma') = \EE(A(\gamma),ec(1,\LL(\gamma);\R(\gamma)),L).$$
\end{lemma}

    \begin{proof} Let $k$ and $\epsilon$ be such that $(1,k,\epsilon) \in \gamma$.  Assume first that $k=\sup M_1$.  Let $\overline{\gamma} = \gamma \setminus \{(1,k,\epsilon)\}$, so that $\gamma'= \overline{\gamma} \cup \{(1,k,-\epsilon\}$.  If $\overline{\gamma} = \varnothing$, then the lemma is obvious, so assume otherwise.

    Consider a nonempty $\T  \in \mathcal{T}(M)$ for some $M$, with $1 \notin M$.  If we set $(\T_1,P_1) = \alpha((\T,1,1))$ and $(\T_2,P_2) = \alpha((\T,1,-1))$, and if $l \in M$ is such that $(l, S_{1,1}) \in \T$, then we have $P(1,\T_1) = \{S_{1,1},S_{1,2}\}$, $P(l,\T_1) = \{S_{2,1},S_{2,2}\}$, $P(1,\T_2) = \{S_{1,1},S_{2,1}\}$, and $P(l,\T_2) = \{S_{1,2},S_{1,2}\}$. For $r \in M \setminus \{l\}$ we have $P(r,\T_1)=P(r,\T_2)$, and thus in particular $P_1=P_2$.  Applying this observation with $\LL(\overline{\gamma})$ substituted for $\T$, we find that $c(1,\LL(\gamma))$ is a closed cycle consisting of two elements, that $\LL(\gamma')= \EE(\LL(\gamma),c(1,\LL(\gamma))),$ and that $\R(\gamma')=\R(\gamma)$, as desired.

    Now suppose that $k \neq \sup M_1$.  Then the lemma follows from the previous case by repeated application of Proposition 2.3.2 (b).

    \end{proof}

\begin{theorem}\label{theorem:AcommutesT}
Let $\{\alpha, \beta\} = \{\alpha'_1, \alpha_3\} \subseteq \Pi'$.
    \begin{tenumerate}
        \item   Suppose $\gamma \in D_\ab^L(\mathcal{S}(M_1,M_2)).$ Then $A(T_\ab^L(\gamma))= T_\ab^L(A(\gamma))$.
        \item Suppose $\gamma \in D_\ab^R(\mathcal{S}(M_1,M_2)).$ Then $A(T_\ab^R(\gamma))= T_\ab^R(A(\gamma))$.
    \end{tenumerate}

\end{theorem}

\begin{remark*} When $\{\alpha, \beta\} = \{\alpha_i, \alpha_{i+1}\}$ with $i \geq 2$, the corresponding theorem was proved as Theorem 2.1.19
\end{remark*}

\noindent
{\it Proof of Theorem 4.\ref{theorem:AcommutesT}.}

As usual, it suffices to prove the first part.  Since $T_\ab^L$ and $T_\ba^L$ are inverses, we may assume that $\alpha=\alpha_3$.  For $j \in M_1$, define $k_j$ and $\epsilon_j$ by $(j,k_j,\epsilon_j) \in \gamma$.  Set $\overline{\gamma} = \gamma \setminus \{(j,k_j,\epsilon_j) \; | \; j \in \{1,2,3\}\}$.  Let $d<e<f$ be such that $\{d,e,f\}=\{k_1,k_2,k_3\}.$  Set $\gamma' = T_\ab^L(\gamma)$.  Then one of the following ten cases holds:
    \begin{tenumerate}
        \item $\gamma \supseteq \{(1,f,-1),(2,e,-1),(3,d,\epsilon)\}$ and $ \gamma' = \overline{\gamma} \cup \{(1,e,1),(2,f,1),(3,d,\epsilon)\}.$
        \item $\gamma \supseteq \{(1,e,-1),(2,d,\epsilon),(3,f,1)\}$ and $ \gamma' = \overline{\gamma} \cup \{(1,e,-1),(2,f,1),(3,d,\epsilon)\}.$
        \item $\gamma \supseteq \{(1,f,-1),(2,d,-1),(3,e,1)\}$ and $ \gamma' = \overline{\gamma} \cup \{(1,d,1),(2,f,1),(3,e,1)\}.$
        \item $\gamma \supseteq \{(1,f,-1),(2,d,1),(3,e,1)\}$ and $ \gamma' = \overline{\gamma} \cup \{(1,d,-1),(2,f,1),(3,e,1)\}.$
        \item $\gamma \supseteq \{(1,d,\epsilon),(2,f,-1),(3,e,1)\}$ and $ \gamma' = \overline{\gamma} \cup \{(1,d,\epsilon),(2,e,1),(3,f,-1)\}.$
    \end{tenumerate}
Cases (6)-(10) are obtained from (1)-(5) by replacing $\gamma$ with $\SC(\gamma')$.  By Remarks 2.1.6-2 and 4.3.2-2 as well as Definition 4.\ref{definition:T13}-4, we need only consider cases (1)-(5).

In case (1), set $\gamma_1 = \overline{\gamma}\cup \{(1,e,-1),(2,f,1),(3,d,\epsilon)\}$.  We observe directly that $A(\gamma)$ satisfies the hypothesis of case (3) of Definition 4.\ref{definition:T13}, that $\LL(\gamma_1)=(\LL(\gamma) \setminus \{(2,S_{2,2}),(3,S_{1,3})\}) \cup \{(2,S_{1,3}),(3,S_{2,2})\}$, and that $\R(\gamma_1)= \R(\gamma)$.  On the other hand, Lemma 4.\ref{lemma:AofSC} says that
$$A(\gamma') = \EE(A(\gamma_1),ec(1,\LL(\gamma_1); \R(\gamma_1)),L).$$  Since $c(3, \LL(\gamma_1)) = c(1, \LL(\gamma_1))$, case (1) is complete.

In case (2), set $\gamma_1 = \overline{\gamma}\cup \{(1,e,1),(2,f,1),(3,d,\epsilon)\}$.  Adopting the notation from Remark 4.\ref{definition:T13}, we note that $H_1 \subseteq \LL(\gamma)$, $\LL(\gamma_1) = (\LL(\gamma) \setminus H_1) \cup H_2$, and that $\R(\gamma_1) = \R(\gamma)$.  As in case (1), we have $A(\gamma') = \EE(A(\gamma_1),ec(3,\LL(\gamma_1); \R(\gamma_1)),L)$ and the rest of the case follows by Remark 4.\ref{definition:T13}.

The proof of case (3) follows the outline of case (1), provided that we let $\gamma_1 = \overline{\gamma}\cup \{(1,d,-1),(2,f,1),(3,e,1)\}$, and case (4) follows as  case (2), letting $\gamma_1 = \overline{\gamma}\cup \{(1,d,1),(2,f,1),(3,e,1)\}$.  In case (5), note that $A(\gamma)$ falls within case (1) of Definition 4.\ref{definition:T13}, $\LL(\gamma') = \In(2,3;\LL(\gamma))$, and $\R(\gamma')= \R(\gamma)$.

\qed

Analogously to Proposition 3.1.2 we have

\begin{lemma}\label{lemma:439} For any tableau $\T$ in the domain of $T_{\alpha_1'\alpha_3}$, there is a c.s.p.b. between $\T$ and $T_{\alpha_1'\alpha_3}\T$.
\end{lemma}



\section{}

We now define a family of operators that extend the generalized $\tau$-invariant to the type $D$ setting.   They were originally introduced in \cite{garfinkle:vogan} in the context of primitive ideals, but the construction therein can be replicated on any set where one has defined a $\tau$-invariant and operators $T_\ab$ based on a branched Dynkin diagram.

The first part of this section describes the action of the operators in the settings of signed permutations and pairs of domino tableaux.  The second shows that these two definitions are compatible by verifying operators on signed permutations and domino tableaux commute with the domino Robinson-Schensted map $A$.    Since in this paper most of the sets in question have left and right $\tau$-invariants as well as left and right operators $T_\ab$, we will similarly speak of objects introduced here as being of left or right type.  We begin by describing sets which will become their domains.

\begin{definition} \label{definition:varioustypes} Suppose either $X=W'$ or $X=W"$ or $X = \mathcal{S}(M_1,M_2)$ with $\{1,2,3,4\} \subseteq M_1$ or $X=\mathcal{T}_D(M_1,M_2)$ with $\{1,2,3,4\} \subseteq M_1$.  Let  $x \in X$ and suppose $\{\beta, \gamma, \delta\} \in \Pi'$ with $\{\beta, \gamma, \delta\} =\{\alpha_1', \alpha_2, \alpha_4\}$.
    \begin{tenumerate}
        \item We say $x$ is of left type $\Ab$ if $\{\beta, \alpha_3\} \subseteq \tau^L(x)$, $\{\gamma, \delta\} \cap \tau^L(x) = \varnothing$, and $T_{\gamma \alpha_3}^L(x) \neq T_{\delta \alpha_3}^L(x).$
        \item We say $x$ is of left type $\Bb$ if $\{\gamma, \delta\} \subseteq \tau^L(x)$, $\{\beta, \alpha_3\} \cap \tau^L(x) = \varnothing$, and $T_{\alpha_3 \gamma}^L(x) \neq T_{\alpha_3 \delta}^L(x).$
        \item We say $x$ is of left type $\mathcal{C}$ if $\{\beta, \gamma, \delta\} \subseteq \tau^L(x)$, $\alpha_3 \notin \tau^L(x)$, and $T_{\alpha_3 \beta}^L(x) =T_{ \alpha_3 \gamma}^L(x) =T_{ \alpha_3 \delta}^L(x). $
        \item We say $x$ is of left type $\mathcal{D}$ if $\alpha_3 \in \tau^L(x)$, $\{\beta, \gamma, \delta\} \cap \tau^L(x) = \varnothing$, and $T_{\beta \alpha_3}^L(x) =T_{\gamma \alpha_3}^L(x) =T_{\delta \alpha_3}^L(x). $
        \item If $x$ is of left type $\Ab$ or $\Bb$ for some $\beta \in \{\alpha_1', \alpha_2, \alpha_4\}$, or it is of left type $\mathcal{C}$ or $\mathcal{D}$, we say that $x$ is of left type $8$.
        \item Suppose that $\{1,2,3,4\} \subseteq M_2$. We define analogously right types $\Ab$, $\Bb$, $\mathcal{C}$, $\mathcal{D}$, and $8$.
    \end{tenumerate}

\end{definition}

\begin{remark*}
   \begin{tenumerate}
        \item  We note that $x$ is of left type $\mathcal{C}$ if and only if $T^L_{\alpha_3 \beta}(x)$ is of left type $\mathcal{D}$.  Also, $x$ is of left type $\Ab$ if and only if $T^L_{\gamma \alpha_3}(x)$ is of left type $\mathcal{B}_\delta$.  Similarly with right in place of left.
        \item If $w \in W'$ or $w\in W"$, then $w$ is of left (resp. right) type $\Ab$ if and only if $\delta(w)$ is of right (resp. left) type $\Ab$, and similarly for the other types.  This follows from Remarks 2.1.7-1 and  4.\ref{definition:tau}-1.
        \item If $\gamma \in \mathcal{S}(M_1,M_2)$, then $\gamma$ is of left (resp. right) type $\Ab$ if and only if $A(\gamma)$ is of left (resp. right) type $\Ab$, and similarly for the other types.  This follows from Propositions 2.1.18, 4.\ref{proposition:alpha1tauRSK}, and Theorems 2.1.19 and 4.\ref{theorem:AcommutesT}.
   \end{tenumerate}
\end{remark*}

In \cite{garfinkle:vogan}, we grouped type 8 modules into sets of ten or fourteen.  Each set consists of either one module each of type $\mathcal{A}_\beta$ and  $\mathcal{B}_{\beta}$  for $\beta \in \{\alpha_1', \alpha_2, \alpha_4\}$ and two each of types $\mathcal{C}$ and $\mathcal{D}$, or two modules each of the types $\mathcal{A}_\beta$ and  $\mathcal{B}_{\beta}$  for $\beta\in \{\alpha_1', \alpha_2, \alpha_4\}$, one module of type $\mathcal{C}$, and one module of type $\mathcal{D}$.  We will implement this grouping by means of some operators, which, to each type $\mathcal{C}$ (resp, type $\mathcal{D}$) object associate a set consisting of one or two type $\Bb$ (resp. $\Ab$) objects for $\beta \in \{\alpha_1', \alpha_2, \alpha_4\}$.  Similarly, we define maps from type $\Ab$ objects to sets of type $\mathcal{D}$ objects, and from type $\Bb$ objects to sets of type $\mathcal{C}$ objects.  We now define these correspondences for the sets considered in Definition 4.\ref{definition:varioustypes}.

\begin{notation} \label{notation:scs1} Let $\gamma \in \mathcal{S}(M_1,M_2)$.
If $\{i-1,i\} \subseteq M_1$, we set $s_{\alpha_i}\gamma=s_i \gamma= \In^L(i-1,i;\gamma).$
If $\{1,2\} \subseteq M_1$, we define $s_{\alpha_1'}\gamma=s_1' \gamma= \SC^L(1; \SC^L(2;(\In^L(1,2;\gamma)))).$
If on the other hand $\{i-1,i\} \subseteq M_2$ or $\{1,2\} \subseteq M_2$, we define $\gamma s_{\alpha_i}$ and $ \gamma s_{\alpha_1'}$  by replacing $L$ with $R$ in the above.
\end{notation}

\begin{remark*}We have $\delta(s_1'\gamma)= \delta(\gamma)s_1'$, $\delta(\gamma s_1') = s_1'\delta(\gamma)$, $ \delta(s_i \gamma) = \delta(\gamma) s_i$, and $\delta(\gamma s_i) = s_i \delta(\gamma)$.
\end{remark*}

\begin{definition} \label{definition:Tforothertypes} Let $w \in W'$ or $w \in \mathcal{S}(M_1,M_2)$ and suppose that we have labeled the roots $\{\beta, \gamma, \delta\}=\{\alpha_1',\alpha_2, \alpha_4\} \subseteq \Pi'$.
    \begin{tenumerate}
        \item If $w$ is of left type $\Ab$, we define
            $$T_{D\beta}^L(w) = \{ y \in \{s_\beta w, s_{\alpha_3} s_\gamma s_\delta w, s_{\alpha_3} s_\delta s_{\alpha_3} w , s_{\alpha_3} s_\gamma s_{\alpha_3} w \} \; | \; y \text{ is of left type $\mathcal{D}$}\}.$$
        \item If $w$ is of left type $\Bb$, we define
         $$T_{C\beta}^L(w) = \{ y \in \{s_\beta w, s_{\alpha_3} s_\gamma s_\delta w, s_{\alpha_3} s_\delta s_{\alpha_3} w , s_{\alpha_3} s_\gamma s_{\alpha_3} w \} \; | \; y \text{ is of left type $\mathcal{C}$}\}.$$
        \item If $w$ is of left type $\mathcal{C}$, we define
         $$T_{\beta C}^L(w) = \{ y \in \{s_\beta w, s_{\alpha_3} s_\gamma s_\delta w, s_{\alpha_3} s_\delta s_{\alpha_3} w , s_{\alpha_3} s_\gamma s_{\alpha_3} w \} \; | \; y \text{ is of left type $\mathcal{\Bb}$}\}.$$
        \item If $w$ is of left type $\mathcal{D}$, we define  $$T_{\beta D}^L(w) = \{ y \in \{s_\beta w, s_{\alpha_3} s_\gamma s_\delta w, s_{\alpha_3} s_\delta s_{\alpha_3} w , s_{\alpha_3} s_\gamma s_{\alpha_3} w \} \; | \; y \text{ is of left type $\mathcal{\Ab}$}\}.$$
        \item If $w$ is of right type $\mathcal{\Ab}$, we define
         $$T_{D \beta }^R(w) = \{ y \in \{w s_\beta , w s_{\alpha_3} s_\gamma s_\delta, w s_{\alpha_3} s_\delta s_{\alpha_3}, w s_{\alpha_3} s_\gamma s_{\alpha_3} \} \; | \; y \text{ is of right type $\mathcal{D}$}\}$$ and proceed to define $ T_{C \beta }^R(w)$, $T_{\beta D }^R(w)$  and  $T_{\beta C}^R(w)$ analogously.
    \end{tenumerate}

\end{definition}

\begin{remark*}
    \begin{tenumerate}
        \item   If $w \in W'$ is of left type $\Ab$, we have $\delta(T_{D\beta}^L(w)) = T_{D\beta}^R((\delta(w))$, and similarly for the other types.
        \item   Let $\gamma \in \mathcal{S}(M_1,M_2)$.  Then $\gamma$ is of left type $\mathcal{C}$ if and only if $\SC(\gamma)$ is of left type $\mathcal{D}$, in which case $\SC(T_{\beta C}^L(\gamma)) = T_{\beta D} (\SC(\gamma))$.  Similarly, $\gamma$ is of left type $\Ab$ if and only if $\SC(\gamma)$ is of left type $\Bb$, in which case $\SC(T_{D\beta }^L(\gamma)) = T_{C \beta} (\SC(\gamma))$. Similarly with right in place of left.
        \item  Let $\gamma \in \mathcal{S}(M_1,M_2)$.  If $\gamma$ is of left type $\Ab$, then $\gamma \in \mathcal{S}'(M_1,M_2)$ if and only if $T_{D\beta}^L(\gamma) \subseteq \mathcal{S}'(M_1,M_2)$, and similarly for the other types.
    \end{tenumerate}
\end{remark*}

\begin{proposition} \label{proposition:1goesto2} Adopt the notation of Definition 4.\ref{definition:Tforothertypes}.
    \begin{tenumerate}
        \item   Let $X=C$ or $D$, and $Y=R$ or $L$.  The sets of the form $T_{\beta X}^Y(w)$ and  $T_{X \beta}^Y(w)$ consist of one or two elements.  Furthermore,
            \begin{enumerate}
                \item[(i)] if $T_{\beta X}^Y(w) = \{w'\}$, then $T_{X \beta}^Y(w')= \{w, w"\}$ with $w" \neq w'$ and $T_{\beta X}^Y(w") = \{w'\}$, and
                \item[(ii)] if $T_{\beta X}^Y(w) = \{w_1, w_2\}$ with $w_1 \neq w_2$, then $T_{X \beta}^Y(w_1) = T_{X \beta}^Y(w_2) = \{w\}$.
            \end{enumerate}
        \item If $w$ is of left type $\mathcal{C}$, then $T_{\gamma D}^L (T_{\alpha_3 \beta}^L(w)) = T_{\alpha_3 \gamma}^L(w))$.  In particular, $|T_{\beta C}(w)| = T_{\gamma C}^L(w)| = T_{\delta C}^L(w)|$, and similarly with $D$ in place of $C$.  If $w$ is of left type $\mathcal{C}$ and if $T_{\beta C}^L(w) = \{w_1,w_2\}$, then
                $$w_2 = (T_{\delta \alpha_3}^L \circ T_{\alpha_3 \beta}^L \circ T_{\gamma \alpha_3}^L \circ T_{ \alpha_3 \delta}^L \circ  T_{\beta \alpha_3}^L \circ T_{\alpha_3 \gamma}^L )(w_1).$$ If $w$ is of left type $\mathcal{D}$ and if $T_{\beta D}^L(w) = \{w_1,w_2\}$, then
                $$w_2 = (T_{\alpha_3 \delta }^L \circ T_{ \beta \alpha_3}^L \circ T_{\alpha_3 \gamma }^L \circ T_{ \delta  \alpha_3}^L \circ  T_{ \alpha_3 \beta}^L \circ T_{\gamma \alpha_3 }^L )(w_1).$$  These statements also hold with right in place of left.
    \end{tenumerate}
\end{proposition}

    \begin{proof}
        For $W'$ when $n=4$, these statements can be verified directly by explicit case-by-case analysis.  For $W' >4$, they can be deduced from the $n=4$ case using \cite[Exercise 3, \S 1]{bourbaki:ch4}.
         Alternately, they follow from Theorem 2.15 and Proposition 4.4 of \cite{garfinkle:vogan}, regarding these results as statements about irreducible highest weight modules.  The statements for $\mathcal{S}(n,n)$ follow from this, using Remarks 2.1.7-1 and 4.\ref{notation:scs1}, and therefore for any $\mathcal{S}(M_1,M_2)$ by means of an obvious bijection.
    \end{proof}

We now proceed to describe the action of the operators defined above on the tableaux pairs $\mathcal{T}_D(M_1,M_2)$.  Our first task is to simplify the characterization of left types in the setting of domino tableaux.

\begin{notation} \label{notation:shapes} We define certain subsets of $\mathcal{F}$ using the natural identification with Young diagrams. For $X \in \{A, B, C, D \}$, let

$$
\raisebox{2ex}{$F_X^1=$ \;}
\begin{small}
\begin{tableau}
:.{}.{}.{} \\
:.{}.{}.{} \\
:.{}\\
:.{}\\
\end{tableau}
\end{small}
\hspace{.5in}
\raisebox{2ex}{$F_X^2=$ \;}
\begin{small}
\begin{tableau}
:.{}.{}.{}.{} \\
:.{}.{}\\
:.{}.{}\\
:;\\
\end{tableau}
\end{small}
$$
When $X \in \{A, B\}$, let
$$
\raisebox{2ex}{$F_X^3=$ \;}
\begin{small}
\begin{tableau}
:.{}.{}.{}.{} \\
:.{}.{} \\
:.{}\\
:.{}\\
\end{tableau}
\end{small}
\hspace{.5in}
\raisebox{2ex}{$F_X^4=$ \;}
\begin{small}
\begin{tableau}
:.{}.{}.{} \\
:.{}.{}.{}\\
:.{}.{}\\
:;\\
\end{tableau}
\end{small}
$$
And when $X \in \{C, D\}$, let
$$
\raisebox{2ex}{$F_X^3=$ \;}
\begin{small}
\begin{tableau}
:.{}.{}.{} \\
:.{}.{} \\
:.{}.{}\\
:.{}\\
\end{tableau}
\end{small}
\hspace{.5in}
\raisebox{2ex}{$F_X^4=$ \;}
\begin{small}
\begin{tableau}
:.{}.{}.{}.{} \\
:.{}.{}.{}\\
:.{}\\
:;\\
\end{tableau}
\end{small}
$$
\end{notation}

We use the above shapes to replace the final condition in the definition of each of the left types.    The following can be verified easily from the definitions:

\begin{proposition}  Let $(\T_1,\T_2) \in \mathcal{T}_D(M_1,M_2)$ and suppose $\{1,2,3,4\} \subseteq M_1$.  Then $(\T_1, \T_2)$ is of left type $\mathcal{A}_{\alpha'_1}$ if and only if $\{\alpha_1', \alpha_3\} \subseteq \tau(\T_1)$, $\{\alpha_2, \alpha_4\} \cap \tau(\T_1) = \varnothing$, and $\bigcup_{k=1}^4 P(k, \T_1) \in \{F_A^i \; | \; 1 \leq i \leq 4\}$.  The obvious analogue of this statement hold for each of the other types.
\end{proposition}

In particular, the proposition implies that $\T_1$ determines the left type of the pair $(\T_1, \T_2)$, allowing us to abuse notation and speak of an individual domino tableau as being of a particular left type itself.  It is clear that whether a domino tableau is of a particular left type can be determined by its first four dominos.  Furthermore, it turns out that specifying a left type for a domino tableau in $\mathcal{T}_D(4)$ together with a compatible shape from 4.\ref{notation:shapes} determines the tableau itself uniquely.  We make this explicit in the following definition.

\begin{definition} We define a list of domino tableaux in $\mathcal{T}_D(4)$, each determined by a specific shape and left type.  For $1  \leq j \leq 4$ and $\beta \in \{\alpha'_1, \alpha_2, \alpha_4\}$ we require tableaux
\begin{align*}
(A_\beta^j, \phi_D) \text{ of type } & \mathcal{A}_{\beta} \text{ and }  \textup{Shape}(A^j_\beta)=F^j_A, \\
(B_\beta^j, \phi_D) \text{ of type } & \mathcal{B}_{\beta} \text{ and }  \textup{Shape}(B^j_\beta)=F^j_B, \\
(C^j, \phi_D) \text{ of type } & \phantom{.}\mathcal{C} \phantom{.} \text{ and }  \textup{Shape}(C^j)=F^j_C, \\
\text{ and \hspace{2ex}}(D^j, \phi_D) \text{ of type } & \phantom{.}  \mathcal{D} \phantom{.} \text{ and }  \textup{Shape}(D^j)=F^j_D.
\end{align*}

If we write $A_1^j$ in lieu of $A_{\alpha_1'}^j$ and condense notation in a similar fashion for the other types and roots, the sought-after tableaux are given by

$$
\raisebox{2ex}{$A_1^1=$ \;}
\begin{small}
\begin{tableau}
:^1^2^4\\
:;;;\\
:^3\\
\end{tableau}
\end{small}
\hspace{.2in}
\raisebox{2ex}{$A_1^2=$ \;}
\begin{small}
\begin{tableau}
:^1^2>4\\
:;;\\
:>3\\
\end{tableau}
\end{small}
\hspace{.2in}
\raisebox{2ex}{$A_1^3=$ \;}
\begin{small}
\begin{tableau}
:^1^2>4\\
:;;\\
:^3\\
\end{tableau}
\end{small}
\hspace{.2in}
\raisebox{2ex}{$A_1^4=$ \;}
\begin{small}
\begin{tableau}
:^1^2^4\\
:;;;\\
:>3\\
\end{tableau}
\end{small}
$$

$$
\raisebox{2ex}{$A_2^1=$ \;}
\begin{small}
\begin{tableau}
:>1^4\\
:>2\\
:^3\\
\end{tableau}
\end{small}
\hspace{.2in}
\raisebox{2ex}{$A_2^2=$ \;}
\begin{small}
\begin{tableau}
:>1>4\\
:>2\\
:>3\\
\end{tableau}
\end{small}
\hspace{.2in}
\raisebox{2ex}{$A_2^3=$ \;}
\begin{small}
\begin{tableau}
:>1>4\\
:>2\\
:^3\\
\end{tableau}
\end{small}
\hspace{.2in}
\raisebox{2ex}{$A_2^4=$ \;}
\begin{small}
\begin{tableau}
:>1^4\\
:>2\\
:>3\\
\end{tableau}
\end{small}
$$

$$
\raisebox{2ex}{$A_4^1=$ \;}
\begin{small}
\begin{tableau}
:^1>2\\
:;>3\\
:^4\\
\end{tableau}
\end{small}
\hspace{.2in}
\raisebox{2ex}{$A_4^2=$ \;}
\begin{small}
\begin{tableau}
:>1>2\\
:>3\\
:>4\\
\end{tableau}
\end{small}
\hspace{.2in}
\raisebox{2ex}{$A_4^3=$ \;}
\begin{small}
\begin{tableau}
:>1>2\\
:>3\\
:^4\\
\end{tableau}
\end{small}
\hspace{.2in}
\raisebox{2ex}{$A_4^4=$ \;}
\begin{small}
\begin{tableau}
:^1>2\\
:;>3\\
:>4\\
\end{tableau}
\end{small}
$$

$$
\raisebox{2ex}{$B_1^1=$ \;}
\begin{small}
\begin{tableau}
:>1^3\\
:>2\\
:^4\\
\end{tableau}
\end{small}
\hspace{.2in}
\raisebox{2ex}{$B_1^2=$ \;}
\begin{small}
\begin{tableau}
:>1>3\\
:>2\\
:>4\\
\end{tableau}
\end{small}
\hspace{.2in}
\raisebox{2ex}{$B_1^3=$ \;}
\begin{small}
\begin{tableau}
:>1>3\\
:>2\\
:^4\\
\end{tableau}
\end{small}
\hspace{.2in}
\raisebox{2ex}{$B_1^4=$ \;}
\begin{small}
\begin{tableau}
:>1^3\\
:>2\\
:>4\\
\end{tableau}
\end{small}
$$

$$
\raisebox{2ex}{$B_2^1=$ \;}
\begin{small}
\begin{tableau}
:^1^2^3\\
:;\\
:^4\\
\end{tableau}
\end{small}
\hspace{.2in}
\raisebox{2ex}{$B_2^2=$ \;}
\begin{small}
\begin{tableau}
:^1^2>3\\
:;\\
:>4\\
\end{tableau}
\end{small}
\hspace{.2in}
\raisebox{2ex}{$B_2^3=$ \;}
\begin{small}
\begin{tableau}
:^1^2>3\\
:;\\
:^4\\
\end{tableau}
\end{small}
\hspace{.2in}
\raisebox{2ex}{$B_2^4=$ \;}
\begin{small}
\begin{tableau}
:^1^2^3\\
:;;\\
:>4\\
\end{tableau}
\end{small}
$$

$$
\raisebox{2ex}{$B_4^1=$ \;}
\begin{small}
\begin{tableau}
:^1^3^4\\
:;\\
:^2\\
\end{tableau}
\end{small}
\hspace{.2in}
\raisebox{2ex}{$B_4^2=$ \;}
\begin{small}
\begin{tableau}
:>1>4\\
:^2^3\\
:;\\
\end{tableau}
\end{small}
\hspace{.2in}
\raisebox{2ex}{$B_4^3=$ \;}
\begin{small}
\begin{tableau}
:^1^3>4\\
:;;\\
:^2\\
\end{tableau}
\end{small}
\hspace{.2in}
\raisebox{2ex}{$B_4^4=$ \;}
\begin{small}
\begin{tableau}
:>1^4\\
:^2^3\\
:;\\
\end{tableau}
\end{small}
$$

$$
\raisebox{2ex}{$C^1=$ \;}
\begin{small}
\begin{tableau}
:^1>3\\
:;>4\\
:^2\\
\end{tableau}
\end{small}
\hspace{.2in}
\raisebox{2ex}{$C^2=$ \;}
\begin{small}
\begin{tableau}
:>1>3\\
:^2^4\\
:;\\
\end{tableau}
\end{small}
\hspace{.2in}
\raisebox{2ex}{$C^3=$ \;}
\begin{small}
\begin{tableau}
:^1>3\\
:;^4\\
:^2\\
\end{tableau}
\end{small}
\hspace{.2in}
\raisebox{2ex}{$C^4=$ \;}
\begin{small}
\begin{tableau}
:>1>3\\
:^2>4\\
:;\\
\end{tableau}
\end{small}
$$

$$
\raisebox{2ex}{$D^1=$ \;}
\begin{small}
\begin{tableau}
:^1>2\\
:;>4\\
:^3\\
\end{tableau}
\end{small}
\hspace{.2in}
\raisebox{2ex}{$D^2=$ \;}
\begin{small}
\begin{tableau}
:>1>2\\
:^3^4\\
:;\\
\end{tableau}
\end{small}
\hspace{.2in}
\raisebox{2ex}{$D^3=$ \;}
\begin{small}
\begin{tableau}
:^1>2\\
:;^4\\
:^3\\
\end{tableau}
\end{small}
\hspace{.2in}
\raisebox{2ex}{$D^4=$ \;}
\begin{small}
\begin{tableau}
:>1>2\\
:^3>4\\
:;\\
\end{tableau}
\end{small}
$$

When $A_1^j \subseteq \mathbf{T}$, we define $\textup{Re}(A,X;\mathbf{T}) = (\mathbf{T}\setminus A) \cup X$, where $X=D^j$.  Analogously, we define the other obvious variations of this.  As usual, we are abusing notation here so that for $\mathbf{T} = (T, \phi_D)$ we write $A_1^j \subseteq \mathbf{T}$ in place of $A_1^j \subseteq T$.

\end{definition}

\begin{lemma} \label{lemma:3inecnot4} Consider  $(\T_1, \T_2) \in \mathcal{T}_D(M_1,M_2)$ and suppose that either $C^j$ or $D^j \subseteq \T_1$ for some $j$.  Write $Y$ for the latter set of dominos.
If $\T_1' = \textup{Re}(Y,X,\T_1)$, where $X \in \{A_1^j, A_2^j, A_4^j, B_1^j, B_2^j, B_4^j\}$, then $$3 \in ec(4, \T_1; \T_2) \text{ if and only if } 3 \notin ec(4,\T_1'; \T_2).$$
\end{lemma}
\begin{proof} This can be observed directly when $|M_1|=4$.  Otherwise, the argument is by induction using Proposition 2.3.3 (b), as in the proof of Remark 2.3.4.
\end{proof}

Analogously to Proposition 3.1.3, we have

\begin{lemma}\label{cspb:typeD} With notation as in the previous lemma, we have a c.s.p.b. between open cycles in $\T_1$ not containing 3 or 4 and open cycles in $\T_1'$ not containing 3 or 4.  This extends to a c.s.p.b. between all the open cycles in $\T_1$ and those in $\T_1'$ if either $c(4,\T_1)$ is closed or $3\in c(4,\T_1)$.  Otherwise the $S_b$ squares of $c(3,\T_1)$ and $c(4,\T_1')$ coincide and likewise for the $S_f$ squares of the same cycles.
\end{lemma}

We are ready to describe the action of the operators of this section on pairs of domino tableaux.

\begin{definition} \label{definition:TalphaC}

Let $(\T_1, \T_2) \in \mathcal{T}_D(M_1,M_2)$ and suppose it is of left type $\mathcal{C}$.  Let $j$ be such that $C^j \subseteq \T_1$.
\begin{tenumerate}
    \item Suppose $4 \in ec(3, \T_1; \T_2).$  Then $j \in \{1,2\}$.  Let $\T_1' = \textup{Re}(C^j, B^j_1; \T_1).$  Set $c_k = ec(k, \T_1'; \T_2)$ for $k \in \{3,4\}$, and note that by Lemma 4.\ref{lemma:3inecnot4}, $c_3 \neq c_4$.
        \begin{enumerate}[label=(\alph*), leftmargin=.6in, topsep=0ex, itemsep=0ex]
            \item If $n_h(C^j) = n_h(B_1^j)$, we define
                    $$T_{\alpha'_1 C}^L((\T_1,\T_2)) = \{(\T_1', \T_2), \E((\T_1',\T_2); \{c_3,c_4\},L)\}.$$
            \item If $n_h(C^j) \neq n_h(B_1^j)$, we define
                    $$T_{\alpha'_1 C}^L((\T_1,\T_2)) = \{\E((\T_1', \T_2),c_3,L), \E((\T_1',\T_2),c_4,L)\}.$$
        \end{enumerate}
    \item Suppose $4 \notin ec(3, \T_1; \T_2).$
            \begin{enumerate}[label=(\alph*), leftmargin=.6in, topsep=0ex, itemsep=0ex]
            \item If $n_h(C^j) = n_h(B_1^j)$ and $j\in \{1,2\}$, we define
                    $$T_{\alpha'_1 C}^L((\T_1,\T_2)) = \{\textup{Re}(C^j,B_1^j;\T_1), \T_2)\} .$$
            \item If $n_h(C^j) \neq n_h(B_1^j)$ and $j\in \{1,2\}$, we define
                    $$T_{\alpha'_1 C}^L((\T_1,\T_2)) = \{\E((\T_1', \T_2),ec(3, \T_1';\T_2),L)\}.$$
            \item If $j \in \{3,4\},$ we let $k \in \{1,2\}$ be such that $n_h(C^k) \neq n_h(B_1^k)$, and we let $l \in \{3,4\}$ be such that $\E(\mathbf{C}^j, c(l, \mathbf{C}^j))= \mathbf{C}^k$, where $\mathbf{C}^j = (C^j, \phi_D)$.  We set $(\T_1',\T_2') = \E((\T_1, \T_2), ec(l, \T_1; \T_2),L)$, and define
                $$T_{\alpha'_1 C}^L((\T_1,\T_2)) = \{\textup{Re}(C^k,B_1^k;\T_1'), \T_2')\} .$$
        \end{enumerate}
\end{tenumerate}
Each of the other maps of this type are defined using obvious analogues of the above.
\end{definition}

\begin{remark*}
\begin{tenumerate}
    \item The statements of Proposition 4.\ref{proposition:1goesto2} hold for the maps defined in Definition 4.\ref{definition:TalphaC}.
    \item Let $Y \in \mathcal{T}_D(M_1,M_2)$.  Then $Y$ is of left type $\mathcal{C}$ if and only if $^tY$ is of left type $\mathcal{D}$. If $Y$ is of left type $\mathcal{C}$, then $^t(T^L_{\beta C}(Y)) = T^L_{\beta D}(^t Y)$.  Similarly, if $Y$ is of left type $\Ab$, then $^t(T^L_{D \beta}(Y)) = T^L_{C \beta}(^tY)$.  Similarly with right in place of left.
    \item Let $Y \in \mathcal{T}_D(M_1,M_2)$.  If $Y$ is of left type $\Ab$, then $Y \in \mathcal{T}'_D(M_1,M_2)$ if and only if $T^L_{D \beta} (Y) \subseteq \mathcal{T}_D'(M_1,M_2)$, and similarly for the other types.  This uses Proposition 4.\ref{proposition:nhmt}.
\end{tenumerate}
\end{remark*}

\begin{proposition}
Let $\gamma \in \mathcal{S}(M_1,M_2)$ and let $\beta \in \{\alpha'_1, \alpha_2, \alpha_4\}$.
\begin{tenumerate}
    \item   \begin{enumerate}[label=(\alph*), leftmargin=.6in, topsep=0ex, itemsep=0ex]
                \item If $\gamma$ is of left type $\Ab$, then $T^L_{D \beta}(A(\gamma)) = A(T^L_{D \beta}(\gamma))$.
                \item If $\gamma$ is of left type $\Bb$, then $T^L_{C \beta}(A(\gamma)) = A(T^L_{C \beta}(\gamma))$.
                \item If $\gamma$ is of left type $\mathcal{C}$, then $T^L_{\beta C}(A(\gamma)) = A(T^L_{ \beta C}(\gamma))$.
                \item If $\gamma$ is of left type $\mathcal{D}$, then $T^L_{\beta D}(A(\gamma)) = A(T^L_{ \beta D}(\gamma))$.
            \end{enumerate}
    \item Analogous statements hold with right in place of left.
\end{tenumerate}

\end{proposition}
\begin{proof}  As usual it suffices to prove (1).  In fact, using Proposition 4.\ref{proposition:1goesto2}, Remark 4.\ref{definition:TalphaC}(1), and Theorems 2.1.19 and 4.\ref{theorem:AcommutesT}, it suffices to prove (1)(d).  Further, it is enough to prove that given $\gamma \in \mathcal{S}(M_1,M_2)$ of left type $\mathcal{D}$, there is an $\eta \in \{\alpha_1',\alpha_2, \alpha_4\}$ and a $\gamma' \in  T^L_{\eta D}(\gamma)$ such that $A(\gamma') \in T^L_{\eta D}(A(\gamma))$.  So assume $\gamma $ is of left type $\mathcal{D}$.  Define for $ j \in M_1$, $k_j$ and $\epsilon_j$ by $(j, k_j, \epsilon_j) \in \gamma$.  Let $d<e<f<g$ be such that $\{d,e,f,g\} = \{k_1,k_2,k_3,k_4\}$.  There are sixteen cases.  We list the first eight:

\begin{tenumerate}
    \item $\gamma \subseteq \{(4,d,\epsilon),(1,e,-1),(2,f,1),(3,g,-1)\}$
    \item $\gamma \subseteq \{(4,d,\epsilon),(1,e,-1),(3,f,-1),(2,g,1)\}$
    \item $\gamma \subseteq \{(1,d,-1),(4,e,1),(2,f,1),(3,g,-1)\}$
    \item $\gamma \subseteq \{(1,d,-1),(4,e,-1),(3,f,-1),(2,g,1)\}$
    \item $\gamma \subseteq \{(1,d,-1),(4,e,1),(3,f,-1),(2,g,1)\}$
    \item $\gamma \subseteq \{(1,d,-1),(3,e,-1),(4,f,1),(2,g,1)\}$
    \item $\gamma \subseteq \{(1,d,1),(4,e,1),(3,f,-1),(2,g,1)\}$
    \item $\gamma \subseteq \{(1,d,1),(3,e,-1),(4,f,1),(2,g,1)\}$
\end{tenumerate}

Cases (9)-(16) are obtained from the above by replacing $\gamma$ with $s_3(\SC(\gamma)) = \SC(s_3 \gamma)$. Using Remarks 2.1.6 (2), 4.\ref{definition:Tforothertypes}(2), and 4.\ref{definition:TalphaC}(2), together with the results cited at the beginning of this proof, we see that is suffices to verfiy the theorem in the first eight cases.  In the first six cases, we observe that $s_2\gamma \in T^L_{\alpha_2 D}(\gamma)$, that $D^1 \subseteq \mathbf{L}(\gamma)$, that $\mathbf{L}(s_2 \gamma) = \textup{Re}(D^1, A_2^1; \mathbf{L}(\gamma))$, and that $\R(s_2 \gamma)=\R(\gamma)$.  That is, we have $A(s_2 \gamma) \in T^L_{\alpha_2 D}(A(\gamma))$, as desired.  In cases (7) and (8), we observe that $D^4 \subseteq \LL(\gamma)$ and set $\gamma_1 = \SC^L(1,\gamma)$.  Then by Lemma 4.\ref{lemma:AofSC} and using the fact that $c(1, \LL(\gamma)) = c(3, \LL(\gamma)),$ we have $A(\gamma_1) = \E(A(\gamma), ec(3, \LL(\gamma);\R(\gamma)), L).$ In particular, $D^1 \subseteq \LL(\gamma_1)$.  Now we observe that $s_1' \gamma \in T^L_{\alpha_1' D}(\gamma),$ that $\LL(s_1' \gamma) = \RE(D^1, A_1^1; \LL(\gamma_1))$, and that $\R(s_1' \gamma) = \R(\gamma_1)$.  That is, we have $A(s_1' \gamma) \in T^L_{\alpha_1' D}(A(\gamma))$, as desired.

\end{proof}

The proof of Proposition 3.1.4 shows that this result extends word for word to type $D$.  Now Proposition 3.1.5 also carries over to type $D$ and applies in addition to the operator $T_{\alpha_1'\alpha_3}$, with the same proof.   So does Proposition 3.1.6, replacing its hypothesis by the hypothesis that  $(\T_1,\T_2)$ be of left type $\mathcal{C}$, the operator $T_{\alpha\beta}$ by $T_{\alpha_1',C}^L$, and replacing the domino labels 1,2 throughout by 3,4, respectively.

\section{}

\begin{definition} We write $\Pi^* = \Pi' \cup \{A,B,C,D\}$ and say that $$\Sigma = (\alpha^1, \beta^1), \ldots, (\alpha^k, \beta^k)$$ is a sequence or $\Pi^*$ if for each $i$ either $\{\alpha^i, \beta^i\}$ is a pair of adjacent simple roots in $\Pi'$ or $\{\alpha^i, \beta^i\}= \{X, \alpha\}$ for some $X \in \{C,D\}$ and $\alpha \in \{\alpha_1', \alpha_2, \alpha_4\}$.  Analogous to Definition 3.2.1, we define the composition of operators $T_\Sigma^L$ and $T_\Sigma^R$ for such a $\Sigma$.  We also carry over to this situation the analogues of any other notation introduced in Definition 3.2.1 and Definition 3.2.5.
\end{definition}

\begin{remark*}  Suppose $\Sigma$ is a sequence for $\Pi^*$ and $(\T_1, \T_2) \in \mathcal{T}_D(M_1,M_2)$ with $M_1 = \{1,\ldots, n\}$.  If $(\T_1',\T_2') \in T_\Sigma^L((\T_1,\T_2))$, then $n_v((\T_1',\T_2')) \equiv n_v((\T_1,\T_2)) \pmod{4}$.
The purpose of this section is to prove an analogue of Theorem 3.2.2.
\end{remark*}

\begin{theorem} \label{theorem:samespecialthensequence} Suppose $(\T_1,\T_2)$ and  $(\T_1',\T_2') \in \mathcal{T}_D(M_1,M_2)$ with $M_1 = \{1, \ldots, n\}$ and suppose $\Ss(\T_2) = \Ss(\T_2')$ and $n_v((\T_1',\T_2')) \equiv n_v((\T_1,\T_2)) \pmod{4}$.  Then there is a sequence $\Sigma$ for $\Pi^*$ such that $(\T_1',\T_2') \in T_\Sigma^L((\T_1,\T_2))$.
\end{theorem}

In broad outline, our proof of Theorem 4.\ref{theorem:samespecialthensequence} will follow that of  Theorem 3.2.2 as closely as possible.  We will prove first Proposition 4.\ref{proposition:458} and Lemma 4.\ref{lemma:459}, the analogues of Proposition 3.2.4 and Lemma 3.2.9.  The latter used as main ingredients Lemmas 3.2.6, 3.2.7, and 3.2.8.  The  analogues of Lemmas 3.2.6 and 3.2.7 obtained by simply replacing $\mathcal{T}_C(M_1,M_2)$ by $\mathcal{T}_D(M_1,M_2)$ and $\Pi$ by $\Pi^*$ are easily verified and we will use them here citing them by their original numbers.  On the other hand, the analogue to Lemma 3.2.8, though valid, is useless because of the parity condition $n_v((\T_1',\T_2')) \equiv n_v((\T_1,\T_2)) \pmod{4}$ in Theorem 4.\ref{theorem:samespecialthensequence}.  We will replace it with a combination of Propositions 4.\ref{proposition:454} and 4.\ref{proposition:456}.  This combination is less powerful than the original, so as a result, the proofs of Proposition 4.\ref{proposition:458} and Lemma 4.\ref{lemma:459} will have many more special cases than the originals.

\begin{definition}  We say a set $F \subseteq \mathcal{F}$ is proper if there is a tableau $T$ in some $\mathcal{T}(M)$ such that $F = \textup{Shape}(T)$.  Let
$$
\raisebox{1ex}{$F_0=$ \;}
\begin{small}
\begin{tableau}
:.{}.{} \\
:.{}.{} \\
\end{tableau}
\end{small}.
$$
We say a set $F \subseteq \mathcal{F}$ is reducible to $F_0$ if $F$ is proper and if either $F=F_0$ or inductively, there is an extremal position $P$ in $F$ such that $F \setminus P$ is reducible to $F_0$.  If $\T \in \mathcal{T}_D(M)$ for some $M$ and if $\textup{Shape}(\T)$ is reducible to $F_0$, then we say that $\T$ is reducible to $F_0$.
\end{definition}

\begin{remark*} If $\T \in \mathcal{T}(M)$ and if $P$ is an extremal position in $T$, then $\textup{Shape}(T) \setminus P$ is proper, since for example $\textup{Shape}(T) \setminus P = \textup{Shape}(T')$, where $(T',v,\epsilon)=\beta((T,P))$.

\end{remark*}

\begin{proposition} \label{proposition:454}
Suppose $\T \in \mathcal{T}_D(M)$ and write $F = \textup{Shape}(\T)$.

\begin{tenumerate}
    \item If $\T$ is special and $S_{2,2} \in \textup{Shape}(\T)$, then $\T$ is reducible to $F_0$.
    \item Suppose that there are sets $P_1, \ldots, P_k$ such that for each $1 \leq j \leq k$, $P_j$ is an extremal position in $F \setminus \bigcup_{r=1}^{j-1} P_r$, and such that $F \setminus \bigcup_{r=1}^k P_r$ is special and contains $S_{2,2}$.  Then $\T$ is reducible to $F_0$.
    \item If $\T$ is special, $S_{2,2} \in \textup{Shape}(\T)$, $P$ is a boxed set in extremal position in $F$, and $S_{2,2} \not \in P$, then $F \setminus P$ is special and is reducible to $F_0$.
\end{tenumerate}

\end{proposition}
\begin{proof}
We prove (1) by induction on $|M|$.  When $|M|=2$ the result is clear as $\textup{Shape}(\T) = F_0$, so assume $|M|>2.$ Let $k = \kappa_1(\T)$.  Since $\T$ is special, $k$ is even and consequently $k \geq 2$.  Let $u = \rho_{k-1}(\T)$.  If $\rho_{k-1}(\T)=\rho_k(\T)$, then we set $P = \{S_{k-1,u}, S_{k,u}\}$.  If $\rho_{k-1}(\T) > \rho_k(\T)$, then the fact that $\T$ special implies that both are odd, so we set $P = \{S_{k-1,u-1},S_{k-1,u}\}$.  Then $P$ is boxed and $\textup{Shape}(\T) \setminus P$ is special.  Finally,  $S_{2,2} \notin P$, since otherwise our choice of $P$ would imply that $|M|=2$.  Hence by induction, we are done.  Statements (2) and (3) follow from (1).

\end{proof}

\begin{proposition}  \label{lemma:455}
Let $\T \in \mathcal{T}_D(M)$ as suppose $\T$ is reducible to $F_0$.  Then there is a $\T' \in \mathcal{T}_D(M)$ such that $\textup{Shape}(\T')=\textup{Shape}(\T)$ and $n_v(\T') \equiv n_v(\T) +2 \pmod{4}$.

\end{proposition}
\begin{proof}
The proof is by induction of $|M|$.  When $|M|=2$, the proposition is obvious, so assume $|M| >2$.  Write $M =\{m_1,\ldots,m_r\}$ with $m_1 < \cdots < m_r$.  Let $P$ be an extremal position in $\T$ such that $\textup{Shape}(\T) \setminus P$ is reducible to $F_0$.  Now let $(\T_1, v, \epsilon) = \beta((\T,P)).$  Let $i$ be such that $v = m_i$, and let $\T_2 \in \mathcal{T}_D(M \setminus \{e\})$ be the tableau obtained from $\T_1$ by replacing $m_{j+1}$ with $m_j$ for each $i<j<r-1$.  Since $\textup{Shape}(\T_2)=\textup{Shape}(\T_1) = \textup{Shape}(\T) \setminus P$, we have that $\T_2$ is reducible to $F_0$, so by induction there is a $\T_3 \in \mathcal{T}_D(M \setminus \{m_r\})$ with $\textup{Shape}(\T_3) = \textup{Shape}(\T_2)$ and $n_v(\T_3) \equiv n_v(\T_2)+2 \pmod{4}.$ Now if $P$ is horizontal and $\epsilon = -1$, or if $p$ is vertical and $\epsilon = 1$, then we set $\T'= \textup{Adj}(\T_2, P, m_r)$.  If $P$ is horizontal and $\epsilon =1$, or if $P$ is vertical and $\epsilon = -1$, then we set $\T'=\textup{Adj}(\T_3, P, m_r)$.  In all cases we have $\textup{Shape}(\T') = \textup{Shape}(\T)$ and Proposition 4.\ref{proposition:nhnv} shows that $n_v(\T') \equiv n_v(\T)+2 \pmod{4}$.

\end{proof}

\begin{proposition} \label{proposition:456}
Let $\T \in \mathcal{T}_D(M)$ and let $P$ be an extremal position in $\T$.  Suppose $\textup{Shape}(\T) \setminus P$ is reducible to $F_0$.  Then there is a $\T' \in \mathcal{T}_D(M)$ such that $P(\sup M, \T') = P$, $\textup{Shape}(\T') = \textup{Shape}(\T)$, and $n_v(\T') \equiv n_v(\T) \pmod{4}$.

\end{proposition}
\begin{proof}
Write $M =\{m_1,\ldots,m_r\}$ with $m_1 < \cdots < m_r$.  Let $(\T_1,v, \epsilon) = \beta((\T,P))$.  Let $i$ be such that $v=m_i$, and let $\T_2 \in \mathcal{T}_D(M \setminus \{m_r\})$ be the tableau obtained from $\T_1$ by replacing $m_{j+1}$ with $m_j$ for each $i \leq j \leq r-1$.  Since $\T \setminus P$ is reducible to $F_0$, we have that $\T_2$ is reducible to $F_0$, so we apply Lemma 4.\ref{lemma:455} to obtain a $\T_3 \in \mathcal{T}_D(M\setminus \{m_r\})$ with $\textup{Shape}(\T_3) = \textup{Shape}(\T_2)$ and $n_v(\T_3) \equiv n_v(\T_2) +2 \pmod{4}.$  If $P$ is horizontal and $\epsilon =1$, or if $P$ is vertical and $\epsilon = -1$, then we set $\T'=\textup{Adj}(\T_2,P,m_r)$.  If $P$ is horizontal and $\epsilon = -1$ or if $P$ is vertical and $\epsilon = 1$, then we set $\T'=\textup{Adj}(\T_3, P, m_r).$  Proposition 4.\ref{proposition:nhnv} shows that $n_v(\T')=n_v(\T)$, and it is clear that our other requirements for $\T'$ are also satisfied.
\end{proof}

More precisely, the only tableau shapes for which the conclusion of Proposition 4.5.6 fails correspond to partitions of the form $(2k+1,3,1^{2l})$ or $(2k,2^2,1^{2l})$ or $(2k,2)$ or $(2^2,1^{2l})$ in exponential notation, for some $k\ge1$ and $l\ge0$. We next prove a special case of Theorem 4.\ref{theorem:samespecialthensequence}.

\begin{lemma} \label{lemma:457}
Suppose $(\T_1,\T_2)$ and $(\T_1',\T_2')$ are pairs of domino tableaux  in  $\mathcal{T}_D(M_1,M_2)$ with $M_1 = \{1, \ldots, n\}$.  Suppose further that $\Ss(\T_2) = \Ss(\T_2')$, $n_v((\T_1',\T_2')) \equiv n_v((\T_1,\T_2)) \pmod{4}$, and $S_{2,2} \notin \textup{Shape}(\T_1)$.  Then there is a sequence $\Sigma $ for $\Pi* \setminus \{\alpha_1'\}$ such that $$(\T_1', \T_2') = T_\Sigma^L((\T_1,\T_2)).$$

\end{lemma}
\begin{proof}
We first note that our hypotheses in fact imply  not only that $\Ss(\T_2) = \Ss(\T_2')$, but in fact, $\T_2'=\T_2$ and that $n_v((\T_1,\T_2))= n_v((\T_1',\T_2'))$.  We proceed by induction on $n$.  Note that when $n=1$ and $n=2$ the result is trivial since in fact we must have $(\T_1',\T_2') = (\T_1,\T_2)$.  So suppose that $n \geq 3$, letting $P = P(n \T_1)$ and $P'=P(n,\T_1')$.

Assume first that $P = P'$.  Using Definition 3.2.5, we let
$$ (\oT_1,\oT_2) = (\T_1,\T_2) - L \text{ \hspace{.2in} and \hspace{.2in}} (\oT_1',\oT_2') = (\T_1',\T_2') - L.$$
Note that $\oT_2'=\oT_2$.  By induction, there is a sequence $\Sigma$ for $\Pi' \setminus \{\alpha_1',\alpha_n\}$ such that $(\oT_1',\oT_2') = T_\Sigma^L((\oT_1,\oT_2))$.  Set $(\T_1^1, \T_2^1) = T_\Sigma^L((\T_1,\T_2))$. Since $\Sigma$ is a sequence for $\Pi' \setminus \{\alpha_1',\alpha_n\}$, we have $P(n,\T_1^1) = P(n, \T_1)$ and $\T_2^1 = \T_2$.  Consequently $(\T_1^1, \T_2^1) = (\T_1',\T_2')$ and the lemma holds in this case.

Now assume $P \neq P'$.  For simplicity we assume further that $P$ is horizontal; the argument is analogous when $P$ is vertical.  Hence $P = \{S_{1,r-1},S_{1,r}\}$ where $r = \rho_1(\T_1)$, and $P' = \{S_{k-1,1},S_{k,1}\}$, where $k = \kappa_1(\T_1)$.  If $r \geq 5$ or if $r=4$ and $k =3$, set $P_1=\{S_{1,r-3},S_{1,r-2}\}$, otherwise set $P_1 = \{S_{k-3,1},S_{k-2,1}\}$.  Let $(\oT_1,\oT_2)=(\T_1,\T_2)-L$.  By inspection, there is a tableau $\oT \in \mathcal{T}_D(n-1)$ such that $\textup{Shape}(\oT) = \textup{Shape}(\oT_1)$, $P(n-1,\oT) = P'$, and $P(n-2,\oT) = P_1$.  It is also clear that $n_v(\oT) = n_v(\oT_1)$.  By induction there is a sequence $\Sigma_1$ for $\Pi' \setminus \{\alpha_1',\alpha_n\}$ such that $(\oT,\oT_2) = T_\Sigma^L((\oT_1,\oT_2))$.  Let $(\T_1^1,\T_2^1) = T_\Sigma^L((\T_1,\T_2))$.  Then $\T_2^1=\T_2$.  Let $\Sigma_2 = (\alpha_n, \alpha_{n-1})$ and let $(\T_1^2,\T_2^2) = T_{\Sigma_2}^L((\T_1^1,\T_2^1)).$  Then $\T_2^2 = \T_2^1$ and $P(n,\T_1^2) = P(n,\T_1')$.  By the case of the lemma we have already proved, there is a sequence $\Sigma_3$ for $\Pi' \setminus \{\alpha_1',\alpha_n\}$ such that $T_{\Sigma_3}^L((\T_1^2,\T_2^2))=(\T_1',\T_2')$.  Setting $\Sigma = \Sigma_1 \Sigma_2 \Sigma_3$ verifies the lemma in this case.

\end{proof}

\begin{proposition} \label{proposition:458}
Let $(\T_1,\T_2) \in \mathcal{T}_D(M_1,M_2)$ with $M_1 = \{1,\ldots, n\}$ and suppose that $S_{2,2} \in \textup{Shape}(\T_1)$.  Then there is a sequence $\Sigma$ for $\Pi^*$ and $(\T_1',\T_2') \in T_\Sigma^L((\T_1,\T_2))$ such that $\T_2'$ is special.
\end{proposition}
\begin{proof}

The proof uses induction on $n$.  Note that the case $n=0$ is trivial and assume by induction that both Theorem 4.\ref{theorem:samespecialthensequence} and Proposition 4.\ref{proposition:458} are true when $M_1=\{1,\ldots, n-1\}$.  We first assume that the square $S_{2,2} \in P(n,\T_1)$.  Let $r = \rho_1(\T_1)$ and $k = \kappa_1(\T_1)$ and suppose that $\T_2$ is not special.  Then one of the following holds:
\begin{tenumerate}
    \item   $P(n,\T_1) = \{S_{2,1}, S_{2,2}\}$,
    \item $P(n,\T_1) = \{S_{2,2}, S_{3,2}\}$ and $r=2$,
    \item $P(n,\T_1) = \{S_{2,2}, S_{3,2}\}$ and $r>2$, or
    \item $P(n,\T_1) = \{S_{2,2}, S_{2,3}\}$ and $r$ is even.
\end{tenumerate}
We first specify $\Sigma$ in a few low rank cases.  If $n=3$ and we are in case (1), then $\Sigma=(\alpha_1',\alpha_3)$ satisfies the requirements of the proposition.  If $n=3$ and we are in case (1), then we set $\Sigma = (\alpha_3,\alpha_1')$.  When $n=4$ and we are in case (3) or (4), then $(\T_1,\T_2) \in D_{\alpha,X}(\mathcal{T}_D(M_1,M_2))$ for some $X \in \{C,D\}$ and every $\alpha \in \{\alpha_1',\alpha_2,\alpha_4\}$.  Either $\Sigma=(\alpha_1',X)$ or $(\alpha_2,X)$ will satisfy the requirement of the proposition.

Now assume that we are not in one of the cases treated above.  The sequence $\Sigma$ will be specified as a composition in several steps. We begin by specifying pairs of squares $P_1$ and $P_2$.  In case (1) and $n \geq 4$, we note that $r \geq 6$ and set $P_1 = \{S_{1,r-1},S_{1,r}\}$ and $P_2 = \{S_{1,r-3},S_{1,r-2}\}$.  In case (2) we set $P_1=\{S_{k-1,1},S_{k,1}\}$ and $P_2=\{S_{k-3,1},S_{k-2,1}\}$.  In cases (3) and (4) and $r \geq 5$ we let $P_1$ and $P_2$ be as in case (1), otherwise let $P_1$ and $P_2$ be as in case (2).

Now let $(\oT_1,\oT_2) = (\T_1,\T_2) - L$ and let $\oT_1^1 \in \mathcal{T}_D(n-1)$ be such that $\textup{Shape}(\oT_1^1)=\textup{Shape}(\oT_1)$, $P(n-1,\oT_1^1) = P_1$, and $P(n-2,\oT_1^1) = P_2$.  It is clear that such a $\oT_1^1$ exists and that $n_v(\oT_1^1) = n_v(\oT_1)$.  Set $\oT_2^1 = \oT_2$.  But now $S_{2,2} \notin \textup{Shape}(\oT_1)$, so by Lemma 4.\ref{lemma:457}, there is a sequence $\Sigma_1$ for $\Pi' \setminus \{\alpha_1', \alpha_n\}$ such that $(\oT_1^1,\oT_2^1) = T_{\Sigma_1}^L((\oT_1,\oT_2)).$  Let $(\T_1^1, \T_2^1) = \T_{\Sigma_1}^L((\T_1,\T_2))$.  Then let $\Sigma_2$ be such that $\Sigma_2 = (\alpha, \beta)$ with $\{\alpha, \beta\} = \{\alpha_{n-1}, \alpha_n\}$ and $T_{\Sigma_2}^L((\T_1^1, \T_2^1)) \neq \varnothing$, and set $(\T_1^2, \T_2^2) = T_{\Sigma_2}^L((\T_1^1, \T_2^1))$.  Let $(\oT_1^2,\oT_2^2) = (\T_1^2, \T_2^2) - L$.  Then $S_{2,2} \in \textup{Shape}(\oT_1^2)$, so by induction there is a sequence $\Sigma_3$ for $\Pi^* \setminus \{\alpha_n\}$ and a $(\oT_1', \oT_2') \in T_{\Sigma_3}^L((\oT_1^2,\oT_2^2))$ such that $\oT_2'$ is special.  Let $(\T_1',\T_2') \in T_{\Sigma_3}^L((\T_1^2, \T_2^2))$ be such that $(\T_1',\T_2') = (\oT_1',\oT_2') - L$.  Then, as is easily seen using Lemma 3.2.7 (1), we have that $\oT_1'$ is special, so setting $\Sigma = \Sigma_1 \Sigma_2 \Sigma_3$ verifies the conclusion of the proposition.

Henceforth assume that $S_{2,2} \notin P(n,\T_1)$.  We begin as in the proof of Proposition 3.2.4, that is, we let $(\oT_1,\oT_2) = (\T_1, \T_2) - L$.  By induction, there is a sequence $\Sigma_1$ for $\Pi^* \setminus \{\alpha_n\}$ and a $(\oT_1^1,\oT_2^1) \in T_{\Sigma_1}^L((\oT_1,\oT_2))$ with $\oT_2^1$ special.  By Lemma 3.2.6, there is a $(\T_1^1, \T_2^1) \in T_{\Sigma_1}^L((\T_1,\T_2))$ such that $(\oT_1^1,\oT_2^1) =(\T_1^1,\T_2^1) -L$.  Now clearly either $\T_2^1$ is special, or $P(n,\T_1^1)$ is unboxed  and $\{n\}$ forms and extended cycle in $\T_1^1$ relative to $\T_2^1$.  In the former case we are done, so we assume the latter.  Set $P =P(n,\T_1^1)$.  Since $S_{2,2}$ is $\phi_D$-fixed, Lemma 3.2.7 (1) implies that $S_{2,2} \notin P.$

The rest of the proof will follow the general outline of Proposition 3.2.4, although lacking Proposition 3.1.10, we will treat separately the situations where $P$ is horizontal and $P$ is vertical. The proof proceeds by a case-by-case analysis. For the most part, the cases we consider parallel those of the proof Proposition 3.2.4 and are numbered accordingly.  In cases (1)-(6), we define positions $P_1$ and $P_2$ and appeal to Propositions 4.\ref{proposition:454} and 4.\ref{proposition:456} to construct a tableau $\oT_1^2 \in \mathcal{T}_D(n-1)$ with $\sh(\oT_1^2) = \sh(\oT_1^1)$, $P(n-1, \oT_1^2)= P_1$, $P(n-2, \oT_1^2)=P_2$, and $n_v(\oT_1^2) \equiv n_v(\oT_1^1) \pmod{4};$ the rest of the proof in these cases will follow as in the proof of Proposition 3.2.4.  Cases (3a) and (7) will parallel case (7) of the proof of Proposition 3.2.4 and will be treated directly.  Finally, cases (3'a) and (6'a) will require a different approach.

Assume first that $P$ is horizontal and set $P=\{S_{ij},S_{i,j+1}\}$.  Then $\phi_D(S_{ij}) = Y$ and both $i$ and $j$ are odd.  First assume $i \geq 3$  and set $r=\rho_{i-2}(\T_1^1)$ and $s= \rho_{i-1}(\T_1^1)$.

\vspace{.1in}

{\it Case 1}.  Here $r>s$.  As in case (1) of the proof of Proposition 3.2.4, we have $r \geq s+2$ and $s \geq j+2$.  We set $P_1 = \{ S_{i-2, r-1}, S_{i-2,r}\}$.  If $i \geq 5$, or if $i=3$ and $s>3$, we set $P_2 = \{S_{i-1,s-1},S_{i-1,s}\}.$  If $i=3$ and $s=3$ and $r >5$, so in fact $r \geq 7$, we let $P_2 = \{S_{i-2,r-3}, S_{i-2,r-2}\}$.  Finally, if $i=3$, $s=3$, and $r=5$, we set $P_2 = \{S_{1,3},S_{2,3}\}$.

\vspace{.1in}

{\it Case 2.} Here $r=s > j+1$.  We set $P_1= \{S_{i-2,r},S_{i-1,r}\}$.  If $i \geq 5$, or if $i=3$ and $r>3$, we set $P_2 =\{S_{i-2,r-1},S_{i-1,r-1}\}$.  If $i=3$ and $r=3$, then $n=4$. If $n_v(\T_1^1) = 3$, we define $P_2$ as before, otherwise let $P_2 = \{S_{2,1}, S_{2,2}\}$.

\vspace{.1in}

{\it Case 3.} Here $r=s=j+1$ and we assume further that either $i \geq 5$ or $i=3$ and $j >1$.  Let $P_1$ and $P_2$ be as in the first part of case (2).

\vspace{.1in}

{\it Case 3a.} Here $r=s=j+1$, $i=3$, and $j=1$ and we can specify the sequence $\Sigma$ more directly. We have that either

$$
\raisebox{2ex}{$\T_1^1=$ \;}
\begin{small}
\begin{tableau}
:>1\\
:>2\\
:>3\\
\end{tableau}
\end{small}
\raisebox{2ex}{\hspace{.2in}  or \hspace{.2in} $\T_1^1=$ \;}
\begin{small}
\begin{tableau}
:^1^2\\
:;\\
:>3\\
\end{tableau}
\end{small}
\raisebox{2ex}{ .}
$$
In the first case, we set $\Sigma_2 = (\alpha_1', \alpha_3)$, in the second $\Sigma_2 = (\alpha_2, \alpha_3),(\alpha_3, \alpha_1')$.  Then $\Sigma= \Sigma_1 \Sigma_2$ satisfies the requirements of the proposition.

\vspace{.1in}

In the next four cases we are still operating under the assumption that $P$ is horizontal, but now $i=1$ and $r$ as well as $s$ will be redefined as needed.

\vspace{.1in}

{\it Case 4}.  Here $S_{3,j-1} \in \sh(\T_1^1)$.  Let $r = \kappa_{j-1}(\T_1^1).$ Then because $\oT_1^1$ is special, we know that $r$ is even, and in particular, $r \geq 4$.  If $j>3$, or if $j=3$ and $r \geq 6$, we let $P_1 = \{S_{r-1,j-1},S_{r,j-1}\}$ and $P_2=\{S_{r-3,j-1},S_{r-2,j-1}\}$.  If $j=3$ and $r=4$, we set $k = \kappa_1(\T_1^1)$ and note that $k$ is even since $\oT_1^1$ is special.  If $k \geq 6$, we set $P_1=\{S_{k-1,1},S_{k,1}\}$ and $P_2=\{S_{3,2},S_{4,2}\}$.  If $k=4$, we set $P_1 = \{S_{4,1},S_{4,2}\}$ and $P_2=\{S_{3,1},S_{3,2}\}$.

\vspace{.1in}

{\it Case 5.} Here $S_{3,j-1} \notin \sh(\T_1^1)$ and $S_{3,j-2} \in \sh(\T_1^1).$  If $j>3$, so that $j \geq 5$, we define $P_1$ and $P_2$ as in case (5) of the proof of Proposition 3.2.4.  If $j=3$, we set $k = \kappa_1(\T_1^1)$ and note that $k$ is again even.  We set $P_1 = \{S_{k-1,1},S_{k,1}\}$.  If $k \geq 6$, we set $P_2= \{S_{k-3,1},S_{k-2,1}\}$.  If $k=4$, then $n =4$, and we have to account for the number of horizontal dominos.  If $n_h(\oT_1^1)=2$, we set $P_2=\{S_{2,1},S_{2,2}\}$; if $n_h(\oT_1^1)=0$, we set $P_2=\{S_{1,2},S_{2,2}\}$.

\vspace{.1in}

{\it Case 6.} Here $S_{3,j-2} \notin \sh(\T_1^1)$ and $j>3$, so that $j \geq 5$.  We set $P_1=\{S_{2,j-2},S_{2,j-1}\}$ and $P_2=\{S_{1,j-2},S_{1,j-1}\}$.

\vspace{.1in}

{\it Case 7.} Here $j=3$ and $S_{3,1} \notin \sh(\T_1^1)$  and we can specify the sequence $\Sigma$ more directly. We have that either

$$
\raisebox{1ex}{$\T_1^1=$ \;}
\begin{small}
\begin{tableau}
:>1>3\\
:>2\\
\end{tableau}
\end{small}
\raisebox{1ex}{\hspace{.1in}  or \hspace{.2in} $\T_1^1=$ \;}
\begin{small}
\begin{tableau}
:^1^2>3\\
:;\\
\end{tableau}
\end{small}
\raisebox{1ex}{ .}
$$
In the first case we set $\Sigma_2 = (\alpha_3,\alpha_2),(\alpha_1',\alpha_3)$, in the second we set $\Sigma_2 = (\alpha_3,\alpha_1')$.  Then $\Sigma = \Sigma_1 \Sigma_2$ satisfies the requirements of the proposition.

\vspace{.1in}

This exhausts the cases when $P$ is horizontal and henceforth we suppose that $P$ is vertical, setting $P=\{S_{ji},S_{j+1,i}\}$.  We note $\phi_D(S_{ji}) = Z$ which under our assumptions implies that both $i$ and $j$ are even.  For the first three cases we will assume that $i \geq 4$ and set $r=\kappa_{i-2}(\T_1^1)$ and $s=\kappa_{i-1}(\T_1^1)$.

\vspace{.1in}

{\it Case 1$'$.}  Here $r>s$, and as in case (1), we have $r \geq s+2$ and $s \geq j+2$.  We set $P_1=\{S_{r-1,i-2},S_{r,i-2}\}$ and $P_2=\{S_{s-1,i-1},S_{s,i-1}\}$.

\vspace{.1in}

{\it Case 2$'$.} Here $r=s>j+1$.  We  define the domino positions $P_1 = \{S_{r,i-2},S_{r,i-1}\}$ and $P_2=\{S_{r-1,i-2},S_{r-1,i-1}\}$.

\vspace{.1in}

{\it Case 3$'$.} Here $r=s=j+1$, and setting $k = \kappa_1(\T_1^1)$ and $t = \rho_1(\T_1^1)$, we require further that either $i \geq 6$, $j \geq 4$, $k >4$, or $t>5$.  If $i \geq 6$ or $j \geq 4$, we let $P_1$ and $P_2$ be as in case 2$'$.  Now assume that $i=4$ and $j=2$.  Since $\oT_1^1$ is special, we know that $k$ is even and $t$ is odd.  If $k>4$, we set $P_1= \{S_{3,2}, S_{3,3}\}$ and $P_2 = \{S_{k-1,1},S_{k,1}\}$.  If $k =4$ and $t > 5$, we let $P_1=\{S_{t-1,1},S_{t,1}\}$ and $P_2=\{S_{t-3,1},S_{t-2,1}\}$.

\vspace{.1in}

{\it Case 3$'$a.} Here $i=4$, $j=2$, $\kappa_1(\T_1^1) = 4$, and $\rho_1(\T_1^1) =5$.  We cannot apply the standard argument to this case as the hypotheses of Proposition 4.\ref{proposition:456} do not hold at one of the steps, so we address it directly.  Let $\oT$ be the tableau:

$$
\raisebox{1ex}{$\oT=$ \;}
\begin{small}
\begin{tableau}
:^1>2>4\\
:;>5\\
:^3>6\\
\end{tableau}
\end{small}
$$

\vspace{.1in}
\noindent
If $n_h(\oT_1^1)=4$, set $(\oT_1^2,\oT_2^2) = (\oT,\oT_2^1)$.  If on the other hand we have $n_h(\oT_1^1)=2$, set
$(\oT_1^2,\oT_2^2) = \EE((\oT,\oT_2^1),ec(1,\oT;\oT_2^1),L).$
Then $\Ss(\oT_2^2) = \Ss(\oT_2^1)$ and $n_v((\oT_1^2,\oT_2^2)) = n_v((\oT_1^1,\oT_2^1))$, where we appeal to Proposition 4.\ref{proposition:nhmt} for the latter if necessary.  By induction on Theorem 4.\ref{theorem:samespecialthensequence}, there is a sequence $\Sigma_2$ for $\Pi^* \setminus \{\alpha_7\}$ such that $(\oT_1^2, \oT_2^2) \in T_{\Sigma_2}^L((\oT_1^1,\oT_2^1))$.  Let $(\T_1^2,\T_2^2) \in T_{\Sigma_2}^L((\T_1^1,\T_2^1))$ be such that $(\oT_1^2,\oT_2^2) = (\T_1^2,\T_2^2) - L$.  Let $\Sigma_3 = (\alpha_7, \alpha_6)$ and define $(\T_1^3,\T_2^3) = T_{\Sigma_3}^L((\T_1^2,\T_2^2))$.  Further, let $(\oT_1^3,\oT_2^3) = (\T_1^3,\T_2^3) - L$.  Then by induction on Proposition 4.\ref{proposition:458}, there is a sequence $\Sigma_4$ for $\Pi^* \setminus \{\alpha_7\}$ and a $(\oT_1',\oT_2') \in T_{\Sigma_4}^L((\oT_1^3,\oT_2^3))$ such that $\oT_2'$ is special.  Let $(\T_1',\T_2') \in T_{\Sigma_4}^L((\T_1^3,\T_2^3))$ be such that $(\oT_1',\oT_2') = (\T_1', \T_2') - L$.  Using Lemma 3.2.7-1 and the fact that $S_{4,2} \notin \textup{Shape}(\T_1')$, we have that $P(7,\T_1')$ is $\{S_{3,2},S_{3,3}\}$, so $\T_1'$ is special.  That is, $\Sigma = \Sigma_1 \Sigma_2 \Sigma_3 \Sigma_4$ satisfies the requirement of the proposition.

\vspace{.1in}

In the remaining cases we have $i=2$.  Since $j$ must be even and we have assumed that $S_{2,2} \notin P(n,\T_1)$ and consequently $S_{2,2} \notin P(n,\T_1^1)$, we have that $j \geq 4$.

\vspace{.1in}

{\it Case 4$'$.}  Here $S_{j-1,4} \in \textup{Shape}(\T_1^1)$.  Let $t = \rho_{j-1}(\T_1^1)$ and note that since $\oT_1^1$ has special shape, $t$ must be odd and consequently $t \geq 5$.  We define the domino positions $P_1= \{ S_{j-1,t-1}, S_{j-1,t}\}$ and $P_2=\{S_{j-1,t-3},S_{j-1,t-2}\}$.

\vspace{.1in}

{\it Case 5$'$.}  Here $S_{j-1,4} \notin \textup{Shape}(\T_1^1)$ but $S_{j-2,4} \in \textup{Shape}(\T_1^1)$. We proceed as in case 5, but transposing everything.

\vspace{.1in}

{\it Case 6$'$.}  Here $S_{j-2,4} \notin \textup{Shape}(\T_1^1)$, and letting $k=\kappa_1(\T_1^1)$ and $t=\rho_1(\T_1^1)$, we further assume that either $j>4$, $j=4$ and $k > 6$, or $j=4$, $k=6$, and $t \geq 5$.  If $j>4$, we set $P= \{S_{j-2,3},S_{j-1,3}\}$ and $P_2=\{S_{j-2,2},S_{j-1,2}\}$.  If $j=4$ and $k>6$, we set $P= \{S_{k-1,1},S_{k,1}\}$ and $P_2=\{S_{k-3,1},S_{k-2,1}\}$.  Finally, if $j=4$, $k=6$, and $t \geq 5$, we set $P= \{S_{1,t-1},S_{1,t}\}$ and $P_2=\{S_{3,2},S_{3,2}\}$.

\vspace{.1in}

{\it Case 6$'$a.}  We treat this case in the same fashion as case 3$'$a.  We begin by defining a domino tableau
$$
\raisebox{5ex}{$\oT=$ \;}
\begin{small}
\begin{tableau}
:^1>3\\
:;^5^6\\
:^2\\
:;\\
:^4\\
\end{tableau}
\end{small}
$$

\vspace{.1in}

It is the image of the transpose of the tableau in case 3$'$a under moving through the cycle containing the domino with label 1.  We can now follow the steps of the proof of case 3$'$a to construct the desired sequence $\Sigma$.

\end{proof}

\begin{lemma}\label{lemma:459} Let $(\T_1,\T_2)$ be as in Theorem 4.\ref{theorem:samespecialthensequence} and suppose that $\T_2$ is special and $S_{2,2} \in \sh(\T_1).$  Let $P'$ be an extremal position in $\T_1$ and suppose that there exists a tableau $\T \in \mathcal{T}_D(n)$ such that $\sh(\T) = \sh(\T_1)$, $P(n,\T) = P'$, and $n_v(\T) \equiv n_v(\T_1) \pmod{4}$.  Then there is a sequence $\Sigma$ for $\Pi^*$ and a $(\T_1',\T_2') \in T_\Sigma^L((\T_1,\T_2))$ such that $P(n,\T_1') = P'$ and $\T_2' = \T_2$.
\end{lemma}
\begin{proof}

As in the proof of Proposition 4.\ref{proposition:458}, we will assume by induction on $n$ that Theorem 4.\ref{theorem:samespecialthensequence} is true when $M= \{1,\ldots, n-1\}$.  Let $P=P(n,\T_1)$.  We may assume that $P' \neq P$.  We first prove the lemma under the additional assumption that both $P$ and $P'$ are boxed.

We will prove this portion of the lemma along the general lines of the corresponding portion of Lemma 3.2.9.
In place of Lemma 3.2.8 used in \cite{garfinkle3}, we will appeal to Propositions 4.\ref{proposition:454} and 4.\ref{proposition:456}.  In what follows, let $(\oT_1,\oT_2) = (\T_1,\T_2) - L$.


A majority of the cases employ the same approach, which we call the standard argument.  In each such case we define a set $P_1$ which does not contain $S_{2,2}$ and proceed as follows.  If $P_1$ is boxed, then by Propositions 4.\ref{proposition:454} and 4.\ref{proposition:456}, there is a tableau $\oT_1^1$ whose shape matches that of $\oT_1$, $P(n-1,\oT_1^1) = P'$, and $P(n-2,\oT_2^1) = P_1$.  By induction on Theorem 4.\ref{theorem:samespecialthensequence}, there is a sequence $\Sigma_1$ for $\Pi^* \setminus \{\alpha_n\}$ and $(\oT_1^1, \oT_2^1) = T_{\Sigma_1}^L((\oT_1,\oT_2))$ with $\oT_2^1=\oT_2$.  Using Lemma 3.2.6,  let $(\T_1^1,\T_2^1) \in T_{\Sigma_1}^L((\T_1,\T_2))$ satisfying $(\oT_1^1,\oT_2^1) = (\T_1^1,\T_2^1) - L$.
Lemma 3.2.7 now allows us to describe $P(n,\T_1^1)$ and conclude $\T_2^1 = \T_2$.
At this point, we can let $\Sigma_2 = (\alpha_{n-1},\alpha_n)$ or $(\alpha_n, \alpha_{n-1})$. Setting $\Sigma= \Sigma_1 \Sigma_2$ and $(\T_1',\T_2') = T_{\Sigma_2}^L((\T_1^1, \T_2^2))$ completes the argument.
If the suggested $P_1$ is not boxed, we let $\widetilde{P}_1$ be the domino position sharing its bottom right square with the bottom right square of $P_1$.  In each case where this occurs, $\widetilde{P}_1$ is boxed and the above argument can be made with $\widetilde{P}_1$ in place of $P_1$.


When the standard argument applies, we will simply provide the appropriate $P_1$.  Otherwise, we will provide additional details.
In the first four cases, we assume that both $P$ and $P'$ are horizontal and set $P=\{S_{i,j},S_{i,j+1}\}$ and $P'=\{S_{k,l},S_{k,l+1}\}$.  Since $P$ and $P'$ are boxed, we have that $j$ and $l$ are both even.

\vspace{.1in}
{\it Case A}.  Here $k=i-1$.  If $i>3$, $i=3$ and $l \geq 6$, or $i=2$ and $j \geq 4$, let $P_1=\{S_{i-1,l-2},S_{i-1,l-1}\}$ and apply the standard argument.  If $i=2$ and $j=2$, we define $P_1$ as above, and modify the standard argument by noting that the fact that the desired tableau $\oT_1^1$ exists can be verified simply by inspection.
When $i=3$ and $l=4$ we examine several possibilities.  For the remainder of this case, let $r=\rho_1(\T_1)$ and $s= \kappa_1(\T_1)$.
  \begin{itemize}
    \item[(a)]  Suppose $s \geq 6$.  Let $P_1 = \{S_{s-1,1},S_{s,1}\}$, $P_2= \{S_{s-3,1},S_{s-2,1}\}$ and $(\oT_1,\oT_2)=(\T_1,\T_2)-L$.
        We begin as in the standard argument and construct a sequence $\Sigma_1$ for $\Pi^* \setminus \{\alpha_n\}$ and $(\T_1^1,\T_2^1) \in T_{\Sigma_1}^L((\T_1,\T_2))$ with $P(n-1,\T_1^1)=P_1$, $P(n-2,\T_1^1)= P_2$, $P(n,\T_1^1) = P$, and $\T_2^1 = \T_2$
        Let $\Sigma_2 = (\alpha_{n-1},\alpha_n)$ and write $(\T_1^2,\T_2^2) = T_{\Sigma_2}^L( \T_1^1,\T_2^1)$.  Following the same procedure, we can find a sequence $\Sigma_3$ for $\Pi^* \setminus \{\alpha_n\}$ and $(\T_1^3,\T_2^3) \in T_{\Sigma_3}^L((\T_1^2,\T_2^2))$ with  $P(n-1,\T_1^3)=P'$,  $P(n-2,\T_1^3)= P_2$, $P(n,\T_1^3)= P_1$ and $\T_2^3 = \T_2$.  Letting $\Sigma_4=(\alpha_{n-1},\alpha_{n})$ and $\Sigma=\Sigma_1 \Sigma_2 \Sigma_3 \Sigma_4$, we obtain the desired result.

    \item[(b)] Suppose $r \geq 7$ and let $P_1 = \{S_{1,r-1},S_{1,r}\}$ and $P_2 = \{S_{1,r-3},S_{1,r-2}\}$.  We argue as in (a) and first construct a sequence $\Sigma_1$ for $\Pi^* \setminus \{\alpha_n\}$ and $(\T_1^1,\T_2^1) \in T_{\Sigma_1}^L((\T_1,\T_2))$ satisfying $P(n,\T_1^1) = P$, $P(n-2, \T_1^1) = P'$, and $P(n-1, \T_1^1) = P_1.$  Let $\Sigma_2 = (\alpha_n, \alpha_{n-1})$ and write $(\T_1^2,\T_2^2) = T_{\Sigma_2}^L((\T_1^1,\T_2^1))$.  Repeating the argument, we construct a sequence $\Sigma_3$ for $\Pi^* \setminus \{\alpha_n\}$ and $(\T_1^3,\T_2^3) \in T_{\Sigma_3}^L((\T_1^2,\T_2^2))$ with $P(n,\T_1^3) = P_1$, $P(n-2, \T_1^3) = P_2$, and $P(n-1, \T_1^3) = P'.$  Letting $\Sigma_4=(\alpha_{n},\alpha_{n-1})$ and $\Sigma=\Sigma_1 \Sigma_2 \Sigma_3 \Sigma_4$, we obtain the desired result.

 \item[(c)] Let $r=5$ and $s=4$.  Define a tableau as follows:
    $$
\raisebox{3ex}{$\oT=$ \;}
\begin{small}
\begin{tableau}
:^1 >3>4 \\
:;>5>6\\
:^2\\
\end{tableau}
\end{small}
$$

\vspace{.1in}

\noindent
and let
$$
 (\oT^1_1,\oT^1_2)=
  \begin{cases}
       (\oT, \oT_2)    \hfill & \text{if $n_v(\oT_1) \equiv n_v(\oT) \pmod{4}$} \\
       \EE((\oT,\oT_2),ec(1,\oT;\oT_2),L) \hfill & \text{ otherwise.} \\
  \end{cases}
$$
\noindent
Then $\Ss(\oT_2^1) = \oT_2$ and furthermore $n_v((\oT_1^1,\oT_2^1)) \equiv n_v((\oT_1,\oT_2)) \pmod{4}$.  Working inductively using Theorem 4.\ref{theorem:samespecialthensequence}, we can find a sequence $\Sigma_1$ for  $\Pi^* \setminus \{\alpha_n\}$ such that
$(\oT_1^1,\oT_2^1) \in T_{\Sigma_1}^L(\oT_1,\oT_2).$  Let $(\T_1^1,\T_2^1) \in T_{\Sigma_1}^L((\T_1,\T_2))$ be such that $(\oT_1^1,\oT_2^1) = (\T_1^1,\T_2^1) - L$.
By Lemma 3.2.7, we know that $P(7, \T_1^1) = P(7,\T_1) $ or $P'(7,\T_1)$.  In the former case let $\Sigma_2=(\alpha_6,\alpha_7)$, otherwise we define $\Sigma_2=(\alpha_6,\alpha_7), (\alpha_4, \alpha_5), (\alpha_4, C)$, and write $(\T_1^2, \T_2^2) = T_{\Sigma_2}^L((\T_1^1,\T_2^1))$.  If $\T_2^2$ is special, then $P(7,\T_1^2) = P'$ and $\Sigma=\Sigma_1 \Sigma_2$.  If not, note that $P(7,\T_1^2)$ is boxed and let $(\oT_1^2, \oT_2^2) = (\T_1^2, \T_2^2) - L$.  Then by Proposition 4.\ref{proposition:458}, there is a sequence $\Sigma_3$ for $\Pi^* \setminus \{\alpha_7\}$ and $(\oT_1^3, \oT_2^3) \in T_{\Sigma_3}^L(\oT_1^2, \oT_2^2)$ with $\oT_2^3$ special.  Let $(\T_1^3, \T_2^3) \in T_{\Sigma_3}^L(\T_1^2, \T_2^2)$ be such that $(\T_1^3, \T_2^3) - L = (\oT_1^3, \oT_2^3)$.  Since $P(7,\T_1^2)$ is boxed, Lemma 3.2.7-4 implies $P(7,\T_1^3) = P'$ and $\Sigma=\Sigma_1 \Sigma_2 \Sigma_3$ gives the desired result.

\end{itemize}

\vspace{.1in}
{\it Case B}.  Here $k<i-1$.  Set $r=\rho_{i-1}(\T_1)$ and note that $r \geq j+1 \geq 3$.  If $i>3$ or $r>3$, then set $P_1=\{S_{i-1,r-1},S_{i-1,r}\}$.  If $i=3,$ $r=3$, and $l\geq 6$, let $P_1= \{S_{1,l-2},S_{1,l-1}\}$.  If $i=3$, $r=3$, and $l=4$, set $P_1=\{S_{1,3},S_{2,3}\}$.


\vspace{.1in}
{\it Case C}.  Here $k=i+1$. If $i \geq 3$, $i=2$ and $j \geq 6$, or $i=1$ and $l \geq 4$, let $P_1 = \{S_{i,j-2},S_{i,j-1}\}.$  If $i=1$ and $l=2$, define $P_1$ in the same way and note that by one of the hypotheses of the lemma the vertical dominos of $\T_1$ must all lie in its first column.

The final cases arise when $i=2$ and $j=4$. Let $r=\rho_1(\T_1)$ and $s= \kappa_1(\T_1)$.  When $s \geq 6$, we proceed as in part (a) of case (A) only using the inverses of $\Sigma_2$ and $\Sigma_4$ instead. If $r \geq 7$ we proceed as in part (b) of case (A), reversing the argument and again using the inverses of $\Sigma_2$ and $\Sigma_4$.  When $r=5$ and $s=4$, we proceed as in part (c) of case (A), this time using the tableau
    $$
\raisebox{3ex}{$\oT=$ \;}
\begin{small}
\begin{tableau}
:^1 >3>4 \\
:;>5\\
:^2>6\\
\end{tableau}
\end{small}
$$

\vspace{.1in}

\noindent
and $\Sigma_2=(\alpha_7, \alpha_6)$ or $\Sigma_2= (\alpha_7, \alpha_6), (\alpha_4,\alpha_5),(\alpha_4,C)$.   If a $\T_2^2$ constructed in this way is special, then $P(7,\T_1^2) = P'$ and we are done.  Otherwise,  we continue as in part (c) defining $\Sigma_3$, this time noting that since $S_{3,2} \not \in \oT_1^3$ as $\oT_2^3$ is special, Lemma 3.2.7-1 forces $P(7, \T_1^3) = P'$.

\vspace{.1in}
{\it Case D}.  Here $k>i+1$.  Let $r=\rho_{k-1}(\T_1)$ and note that $r \geq l+1 \geq 3$.  If either $k \neq 3$ or $ r \neq 3$, let
$P_1=\{S_{k-1,r-1},S_{k-1,r}\}$.  If $k=3$, $r=3$, and $j>4$, then let $P_1=\{S_{1,j-2},S_{1,j-1}\}$.  If $k=3$, $r=3$, and $j=4$, we set $P_1=\{S_{1,3},S_{2,3}\}$.


\vspace{.1in}

In the next six cases we assume that $P$ is horizontal and $P'$ is vertical, setting $P=\{S_{ij},S_{i,j+1}\}$ and $P'=\{S_{kl},S_{k+1,l}\}$.  Since $P$ and $P'$ are boxed, we have that $j$ is even and $k$ is odd.

\vspace{.1in}
{\it Case E}. Here $k+1=i-1$.  We set $P_1=\{S_{i-2,l-1},S_{i-1,l-1}\}$.

\vspace{.1in}
{\it Case F}. Here $k+1<i-1$.  Let $r=\rho_{i-1}(\T_1)$ and set $P_1=\{S_{i-1,r-1},S_{i-1,r}\}$.

\vspace{.1in}
{\it Case G}. Here $k=i+1$ and $l < j-1$.  If $i \neq 2$ or $j \neq 4$, set $P_1=\{S_{i,j-2},S_{i,j-1}\}$.  When $i=2$ and $j=4$ we examine several possibilities which parallel the subcases of case (A).  In what follows, let $r=\rho_1(\T_1)$ and $s= \kappa_1(\T_1)$.

 \begin{enumerate}
    \item[(a)] If $r \geq 7$, let $P^*=\{S_{1,r-1},S_{1,r}\}$.  We can first use case (A) to move $n$ to $P^*$ and then employ the standard argument with $P_1= P$ to move $n$ to $P'$.

    \item[(b)] If $s \geq 6$, we let $P_1=\{S_{s-1,1}, S_{s,1}\}$ and $P_2=\{S_{s-3,1},S_{s-2,1}\}$.  Arguing as usual, we find a sequence $\Sigma_1$ for $\Pi^* \setminus \{\alpha_n\}$ and $(\T_1^1,\T_2^1) \in T_{\Sigma_1}^L(\T_1,\T_2)$ with $P(n-2,\T_1^1) = P'$, $P(n-1,\T_1^1) = P_1$, and $P(n-2,\T_1^1) = P'$.  Let $\Sigma_2=(\alpha_n,\alpha_{n-1})$ and $(\T_1^2,\T_2^2) = T_{\Sigma_2}^L((\T_1^1,\T_2^1))$.  Continuing, we can find a sequence $\Sigma_3$ for $\Pi^* \setminus \{\alpha_n\}$ and $(\T_1^3,\T_2^3) \in T_{\Sigma_3}^L(\T_1^2,\T_2^2)$ with $P(n-2,\T_1^3) = P_2$, $P(n-1,\T_1^3) = P'$, and $P(n,\T_1^3) = P_1$.  Letting $\Sigma_4=(\alpha_{n-1},\alpha_n)$ and $\Sigma = \Sigma_1 \Sigma_2 \Sigma_3 \Sigma_4$ completes the proof.

    \item[(c)] If $r=5$ and $s=4$, define the tableau $\oT$ by

    $$
\raisebox{5ex}{$\oT=$ \;}
\begin{small}
\begin{tableau}
:^1>{2}>{3}\\
:;>{4}\\
:^5\\
:;\\
\end{tableau}
\end{small}
\raisebox{5ex}{\hspace{.2in} or \hspace{.2in} $\oT=$ \hspace{.2in}}
\begin{small}
\begin{tableau}
:^1>{3}>{4}\\
:;>{5}\\
:^{2}^6\\
:;\\
\end{tableau}
\end{small}
$$

\noindent
if $l=1$ or $l=2$, respectively.  In either case, if  $n_v(\oT_1) \equiv n_v(\oT) \pmod{4}$
we can use a version of the standard argument with $\Sigma_2 = (\alpha_n,\alpha_{n-1})$.  So suppose $n_v(\oT_1) \not \equiv n_v(\oT) \pmod{4}$ and let
$(\oT^1_1,\oT^1_2) = \EE((\oT,\oT_2),ec(1,\oT;\oT_2),L)$.   We first examine $l=1$, noting that $S_{2,3} \not \in \oT_1^1$.  As usual,  we can find a sequence $\Sigma_1$ for  $\Pi^* \setminus \{\alpha_n\}$ such that
$(\oT_1^1,\oT_2^1) \in T_{\Sigma_1}^L(\oT_1,\oT_2).$  Let $(\T_1^1,\T_2^1) \in T_{\Sigma_1}^L((\T_1,\T_2))$ be such that $(\oT_1^1,\oT_2^1) = (\T_1^1,\T_2^1) - L$. By Lemma 3.2.7, we know that $P(6, \T_1^1) = \{S_{2,3},S_{2,4}\}$. If $P(5,\T_1^1)$ is vertical, let the sequence $\Sigma_2 = (\alpha_6,\alpha_5)$, otherwise let
$\Sigma_2 = (\alpha_6,\alpha_5), (\alpha_3,\alpha_4), (\alpha_1',\alpha_3)$ and write $(\T_1^2,\T_2^2) = T_{\Sigma_2}^L((\T_1^1,\T_2^1))$.  In the latter case, $\T_2^2 = \T_2$ and $P(6,\T_1^2) = P'$, as desired.  In the former, $\T_2^2$ is not special and we continue as in part (c) of case (A), again using Lemma 3.2.7-4 to show $P(6,\T_1^3) = P'$.  When $l=2$, we argue as above but let $\Sigma_2 =(\alpha_7,\alpha_6),(\alpha_4, \alpha_5), (\alpha_4,C)$, noting that for one $(\T_1^2,\T_2^2) \in T_{\Sigma_2}^L((\T_1^1,\T_2^1))$ we have $\T_2^2 = \T_2$ and $P(7,\T_1^2) = P'$, as desired.
    \end{enumerate}


\vspace{.1in}
{\it Case H}. Here $k>i+1$.  Let $s=\rho_{i+1}(\T_1)$.  We first consider the case $s\leq j-3$.   If $i=2$ and $j=4$, then $l=1$ and $k \geq 5$ and we set $P_1=\{S_{k-2,1},S_{k-1,1}\}$; otherwise we set $P_1=\{S_{i,j-2},S_{i,j-1}\}$.    If $s=j-1$, let $r=\kappa_{j-1}(\sh (\T_1) \setminus P')$ and set $P_1=\{S_{r-1,j-1},S_{r,j-1}\}$.  Finally, suppose $s=j-2$.  Since $\T_1$ is special and $P$ is boxed, we must have $\phi_D(S_{ij})= Z$ and further $r=\kappa_{j-2}(\sh(\T_1) \setminus P')$ is even and greater or equal to $i+2$.  We can let $P_1=\{S_{r-1,j-2},S_{r,j-2}\}$.


\vspace{.1in}
{\it Case I}.  Here $k=i-1$, implying $P \cap P' = \{S_{i,j+1}\}$.  Since both $P$ and $P'$ are boxed, we have that $\phi_D(S_{ij}) = Z$ and that both $i$ and $j$ are even.  We consider several cases.  After specifying domino positions $P_1$ and $P_2$, or $P''$ the proof follows the outline of the corresponding case of Lemma 3.2.9 unless otherwise indicated.

\begin{itemize}
    \item[(a)]  Here $i \geq 4$.  Let $r$ and $s$ be the lengths of the $i-2$ and $i-3$ rows of $\T_1$ respectively, and assume first that $r=j+1$, and $i \geq 6$ or $j \geq 4$.  Set $P_1= \{S_{i-2.j+1},S_{i-1,j+1}\}$ and $P_2=\{S_{i-2,j},S_{i-1,j}\}$.  If $r=j+1$, $i =4$, $j=2$, and $\kappa_{1}(\T_1) =4$, we set $P_1= \{S_{3,2},S_{3,3}\}$ and $P_2=\{S_{3,1},S_{4,1}\}$.   If $s=r>j+1$ or $s>r$, we define $P''$ as in Lemma 3.2.9.

         If $r=j+1$, $i =4$, $j=2$, and $\kappa_{1}(\T_1) \geq 6$, a more elaborate argument is required.  Let $P_1=\{S_{2,3},S_{3,3}\}$ and $P_2=\{S_{2,2},S_{3,2}\}$. Using the usual notation, let $\oT$ be a tableau of the same shape as $\oT_1$ with $P(n-1,\oT) = P_1$ and $P(n-2,\oT) = P_2$. If $n_v(\oT) \equiv n_v(\oT_1) \pmod{4}$, then we can continue as above.  Otherwise let $(\oT_1^1, \oT_2^1) = \EE((\oT,\oT_2),ec(1,\oT;\oT_2),L)$.  Construct $\Sigma_1$ and $(\T_1^1, \T_2^1)$ as usual, let $\Sigma_2=(\alpha_{n-1},\alpha_n)$ and $(\T_1^2, \T_2^2) = T_{\Sigma_2}^L((\T_1^1, \T_2^1))$.  The tableau $\T_2^2$ is not special, so we argue as usual in this situation and find $\Sigma_3$ and $(\T_1^3, \T_2^3) \in T_{\Sigma_3}^L((\T_1^2, \T_2^2))$ with $\T_2^3$ special and $P(n,\T_1^3) = P'$.

    \item[(b)] Here $ j\geq 4$ and the proof is parallel to that of the corresponding case in Lemma 3.2.9.

    \item[(c)]  Here $i=2$ and $j=2$.  If $t= \kappa_1(\T_1) \geq 6$, let $P''=\{S_{t-1,1},S_{t,1}\}$.  We can use cases (H) and then (D$'$) to move $n$ first to $P''$ and then to $P'$.  If $t=4$, then $n=4$.  Arguing as usual, we can find a sequence $\Sigma_1$ such that $(\T_1^1, \T_2^1) \in T_{\Sigma_1}^L((\T_1,\T_2))$ where
            $$
\raisebox{5ex}{$\T_1^1=$ \;}
\begin{small}
\begin{tableau}
:^1>{2}\\
:;>{4}\\
:^3\\
:;\\
\end{tableau}
\end{small}
 $$
Applying $\Sigma_2=(\alpha_2,D)$  achieves the desired result.  If $t=2$, then $n=3$ and it is sufficient to let $\Sigma = (\alpha_2, \alpha_3)$.
\end{itemize}


\vspace{.1in}
{\it Case J}. Here $k=i+1$ and $l=j-1$ and the proof follows the outline of the corresponding case of Lemma 3.2.9.

\vspace{.1in}

In the next four cases we assume that both $P$ and $P'$ are vertical and set $P=\{S_{ji},S_{j+1,i}\}$ and $P'=\{S_{l,k},S_{l+1,k}\}$.  Since $P$ and $P'$ are boxed, we know that both $j$ and $l$ are odd.

\vspace{.1in}
{\it Case A$'$}.  Here $k=i-1$.  If $i \neq 3$ or $l \neq 3$, let $P_1=\{S_{l-2,k},S_{l-1,k}\}$.  If $i=3$ and $l=3$, again set $P_1=\{S_{l-2,k},S_{l-1,k}\}$, define $(\oT_1, \oT_2) = (\T_1,\T_2)-L$, and let $\oT$ be a tableau of the same shape as $\oT_1$ but with $P(n-1,\oT_1)=P'$ and $P(n-2,\oT_1)=P_1$.
If $n_v(\oT) \equiv n_v(\oT_1) \pmod{4}$, the standard argument applies.   If $n_v(\oT) \not \equiv n_v(\oT_1) \pmod{4}$, there are two cases.

\begin{enumerate}
    \item[(a)] If $\kappa_1(\T_1) \geq 6 $, let $(\oT_1^1,\oT_2^1) = \EE((\oT,\oT_2, ec(1,\oT;\oT_2),L)$.  Let $\Sigma_1$ be a sequence for $\Pi^* \setminus \{\alpha_n\}$ so that $(\oT_1^1,\oT_2^1) \in T_{\Sigma_1}^L((\oT_1,\oT_2))$ and let $(\T_1^1,\T_2^1) \in T_{\Sigma_1}^L((\T_1,\T_2))$ satisfy $(\T_1^1,\T_2^1) -L=(\oT_1^1,\oT_2^1)$.  Then $P(n,\T_1^1)=P(n,\T_1)$ or $P'(n,\T_1)$.  Let $\Sigma_2=(\alpha_n,\alpha_{n-1})$ and define  $(\T_1^2,\T_2^2) = T_{\Sigma_2}^L((\T_1^1,\T_2^1))$.
Since $\T_2^2$ is not special, we argue by induction.  Consider $(\oT_1^2, \oT_2^2) = (\T_1^2,\T_2^2)-L$ and let $\Sigma_3$ be a sequence for $\Pi^* \setminus \{\alpha_n\}$ such that $(\oT_1^3, \oT_2^3) \in T_{\Sigma_3}^L((\oT_1^2, \oT_2^2))$ is special.  Let $(\T_1^3,\T_2^3) \in T_{\Sigma_3}^L((\T_1^2,\T_2^2))$ satisfy $(\T_1^3,\T_2^3) -L=(\oT_1^3,\oT_2^3)$.  Then $P(n,\T_1^3)=P'$ and $\T_2^3 = \T_2$.

 \item[(b)] If $\kappa_1(\T_1) = 4$, arguing as usual we can find a sequence $\Sigma_1$ such that $(\T_1^1, \T_2^1) \in T_{\Sigma_1}^L((\T_1,\T_2))$ where

$$
\raisebox{5ex}{$\T_1^1=$ \;}
\begin{small}
\begin{tableau}
:^1^2^5\\
:;\\
:>{3}\\
:>{4}\\
\end{tableau}
\end{small}
$$

\noindent
After applying $\Sigma_2=(\alpha_5,\alpha_4)$ we are in case (I) above.  Letting $\Sigma_3$ be the sequence of operators constructed therein, we define $\Sigma = \Sigma_1 \Sigma_2 \Sigma_3$.
\end{enumerate}

\vspace{.1in}
{\it Case B$'$}.   Here $k<i-1$.  Let $r=\kappa_{i-1}(\T_1)$.  If $i=3$, $r=2$, and $l=3$, then  $\Sigma$ is easy to find by inspection.  If $i=3$, $r=2$, and $l \geq 5$, let $P_1=\{S_{l-2,1},S_{l-1,1}\}$.  If $i=4$ and $r=3$, then $\kappa_2(\T_1)=3$ since $\T_1$ is special, $k=1$, and we let $P_1=\{S_{3,2},S_{3,3}\}$.  In all other cases, we let $P_1=\{S_{r-1,i-1},S_{r,i-1}\}$.

\vspace{.1in}
{\it Case C$'$}.  Here $k=i+1$.  If either $i \geq 3$, $i=2$ and $j \geq 5$, or $i=1$ and $l \geq 3$, let $P_1=\{S_{j-2,i},S_{j-1,i}\}$. If $i=1$ and $l=1$, then the hypotheses  of the lemma allow us to consider just those tableaux $\T_1$ all of whose dominos are vertical; we can define $P_1$ in the same way. If $i=2$ and $j=3$, we consider two cases:
    \begin{enumerate}
        \item[(a)] If $t=\kappa_1(\T_1) \geq 6$, let $P''=\{S_{t-1,1},S_{t,1}\}$.  Using case (A$'$) we can first move $n$ to $P''$ and then to $P'$ using case (D$'$).
        \item[(b)] If $\kappa_1(\T_1) =4$ and $n_v(\T_1) =5$, we can let $P_1=\{S_{1,2},S_{2,2}\}$.  Otherwise, arguing as usual we can find a sequence $\Sigma_1$ such that $(\T_1^1, \T_2^1) \in T_{\Sigma_1}^L((\T_1,\T_2))$ where

$$
\raisebox{2ex}{$\T_1^1=$ \;}
\begin{small}
\begin{tableau}
:>1^4\\
:>2\\
:^3^5\\
\end{tableau}
\end{small}
$$

\vspace{.1in}
\noindent
Applying $\Sigma_2=(\alpha_1',\alpha_3),(\alpha_4,\alpha_5)$ completes the argument.
    \end{enumerate}


\vspace{.1in}
{\it Case D$'$}. Here $k>i+1$.  Let $r= \kappa_{k-1}(\T_1)$.  If $k=3$, $r=2$, and $j=3$, then let $\Sigma=(\alpha_3,\alpha_4)$.  If $k=3$, $r=2$, and $j \geq 5$, set $P_1=\{S_{j-2,1},S_{j-1,1}\}$.  If $k=4$ and $r=3$, then we also must have $\kappa_2(\T_1) =3$;  we set $P_
1=\{S_{3,2},S_{3,3}\}$.  In all other cases, we let $P_1=\{S_{r-1,k-1},S_{r,k-1}\}$.


\vspace{.1in}

In the next six cases we assume that $P$  is vertical and $P'$ is horizonal, setting $P=\{S_{j,i},S_{j+1,i}\}$ and $P'=\{S_{l,k},S_{l,k+1}\}$.  Since both are boxed, we have that $j$ is odd and $k$ is even.

\vspace{.1in}
{\it Case E$'$}.  Here $k+1=i-1$.  We let $P_1=\{S_{l-2,i-1},S_{l-1,i-1}\}$.


\vspace{.1in}
{\it Case F$'$}.  Here $k+1<i-1$.  We let $r=\kappa_{i-1}(\T_1)$ and define $P_1=\{S_{r-1,i-1},S_{r,i-1}\}$.


\vspace{.1in}
{\it Case G$'$}.  Here $k=i+1$ and $l<j-1$.  Note first that $i$ is odd.  If $i \geq 3$ or $i=1$ and $l>2$, let $P_1 =\{S_{j-2,i},S_{j-1,i}\}$.
The case $i=1$ and $l=1$ is excluded by the hypotheses of this lemma. If $i=1$ and $l=2$, note that by a hypothesis of the lemma, we need only to consider those $\T_1$ where $S_{2,2}$ is contained in a horizontal domino.  Define $P_1$ as above.  Arguing as usual we can find a sequence $\Sigma_1$ such that $(\T_1^1, \T_2^1) \in T_{\Sigma_1}^L((\T_1,\T_2))$ with $P(n-2,\T_1^1) = P_1$, $P(n-1,\T_1^1) = P'$, and $P(n,\T_1^1) = P$.  Applying $\Sigma_2=(\alpha_{n-1},\alpha_n)$ completes the argument.


\vspace{.1in}
{\it Case H$'$}.    Here $k > i+1$.  Let $t = \kappa_{i+1}(\T_1)$.  We consider three possibilities.
    \begin{itemize}
    \item[(a)] Here $ t\leq j-3$.   Let $P_1=\{S_{j-2,i},S_{j-1,i}\}$.


    \item[(b)] Here $t= j-1$.  Let $r = \rho_{j-1}(\sh(\T_1) \setminus P')$.  If $j \geq 5$, or if $j =3$ and $r \geq 4$,  let $P_1=\{S_{j-1,r-1},S_{j-1,r}\}$.  Now consider $j=3$ and $r=3$.  If $l=1$, let $P_1=\{S_{1,k-2},S_{1,k-1}\}$ if $k>4$ and  $P_1=\{S_{1,3},S_{2,3}\}$ if $k=4$.  If $l=2$ and $s=\rho_1(\T_1) \geq 7$, let $P''=\{S_{1,s-1},S_{1,s}\}$.  By using the $l=1$ subcase above, we can first move $n$ to $P''$ and then use case (C) to move it to $P'$.  If $l=2$ and $u=\kappa_1(\T_1) \geq 6$, let $P''=\{S_{u-1,1},S_{u,1}\}$.  Using case (A$'$) we can move $n$ to $P''$ and use the $j \geq 5$ subcase above to move it to $P'$.  Finally, assume $l=2$, $u=4$, and $s=5$.  Let

    $$
\raisebox{5ex}{$\oT=$ \;}
\begin{small}
\begin{tableau}
:^1>{2}>{3}\\
:;>{4}>5\\
:;\\
\end{tableau}
\end{small}
\raisebox{5ex}{\hspace{.2in} or \hspace{.2in} $\oT=$ \hspace{.2in}}
\begin{small}
\begin{tableau}
:^1>{3}>{4}\\
:;>{5}>6\\
:^{2}\\
:;\\
\end{tableau}
\end{small}
$$

\noindent
if $i=1$ or $i=2$, respectively.  In either case, if  $n_v(\oT_1) \equiv n_v(\oT) \pmod{4}$
we can argue along the lines of the standard argument with $\Sigma_2 = (\alpha_{n-1},\alpha_{n})$.  So suppose $n_v(\oT_1) \not \equiv n_v(\oT) \pmod{4}$ and let $(\oT^1_1,\oT^1_2) = \EE((\oT,\oT_2),ec(1,\oT;\oT_2),L)$.  As usual,  we can find a sequence $\Sigma_1$ for  $\Pi^* \setminus \{\alpha_n\}$ such that
$(\oT_1^1,\oT_2^1) \in T_{\Sigma_1}^L(\oT_1,\oT_2).$  Let $(\T_1^1,\T_2^1) \in T_{\Sigma_1}^L((\T_1,\T_2))$ be such that $(\oT_1^1,\oT_2^1) = (\T_1^1,\T_2^1) - L$.  Define $\Sigma_2=(\alpha_{n-1},\alpha_n)$ and $(\T_1^2,\T_2^2)= T_{\Sigma_2}^L((\T_1^1,\T_2^1))$.  Then $\T_2^2$ is not special.  Let $(\oT_1^2,\oT_2^2) = (\T_1^2, \T_2^2)-L$.
Using induction, let $\Sigma_3$ be a sequence for  $\Pi^* \setminus \{\alpha_n\}$ so that $(\oT_1^3,\oT_2^3) \in T_{\Sigma_3}^L((\oT_1^2, \oT_2^2))$ is special.  If $(\T_1^3,\T_2^3) \in T_{\Sigma_3}^L((\T_1^2,\T_2^2))$
satisfies    $(\T_1^3,\T_2^3) -L = (\oT_1^3,\oT_2^3)$, then $P(n,\T_1^3) = P'$ and $\T_2^3=\T_2$.


    \item[(c)] Here $t=j-2$.  Note that $\phi_D(S_{ji})= Y$ since $\T_1$ is special and $P'$ is boxed.  Because $P'$ is horizontal, $r=\rho_{j-2}(\sh(\T_1) \setminus P') \geq i+2$ and we set $P_1=\{S_{j-2,r-1},S_{j-2,r}\}$.

\end{itemize}

\vspace{.1in}
{\it Cases I$'$}. Here $k=i-1$.  This case parallels case (I), albeit with fewer special subcases. We omit the arguments.

{\it Cases J$'$}. Here $k=i+1$ and $l=j-1$ and note that since $P$ and $P'$ are boxed, $j \geq 3$.  When $i \geq 3$, the proof parallels that of case (J).  When $i=1$, let $r= \rho_{l-1}(\T_1)$.  If $r = 3$, we let $P''=\{S_{l-1,3},S_{l,3}\}$.  If $r >3$, then $r \geq 5$ and we let $P''=\{S_{l-1,r-1},S_{l-1,r}\}$.  In both cases $P''$ is boxed and we argue as in the final subcases of part (a) of case (I).

This exhausts all of the possible cases, concluding the proof of the lemma when both $P$ and $P'$ are boxed.

We now assume that $P'$ is boxed and $P$ is unboxed.  We will closely follow the proof of Lemma 3.2.9.  First assume that $P$ is horizontal, setting $P=\{S_{ij},S_{i,j+1}\}$. Under these hypotheses, we have $\phi_D(S_{ij}) = W$ or $Y$, but since $\T_1$ is special, it follows that in fact $\phi_D(S_{ij}) = W$ and $i \geq 2$. Furthermore, $\rho_{i-1}(\T_1) = j+1$ for otherwise $S_{i,j+2}$ would form an empty hole.  Set $(\oT_1,\oT_2) = (\T_1,\T_2) -L$.   By Lemma 4.\ref{proposition:458} or by inspection, there is a sequence $\Sigma_1$ for $\Pi^* \setminus \{\alpha_n\}$ and a $(\oT^1_1,\oT^1_2) \in T_{\Sigma_1}^L((\oT_1,\oT_2))$ such that $\oT_2^1$ is special.  Let $(\T_1^1, \T_2^1) \in T_{\Sigma_1}^L((\T_1,\T_2))$ be such that $(\oT_1^1,\oT_2^1)=(\T_1^1, \T_2^1)-L$.  Then $\sh(\oT_1^1) = (\sh(\oT_1) \setminus \{S_{i-1,j+1}\}) \cup \{S_{ij}\}$, and from  Lemma 3.2.7-1, we can conclude that $$P(n,\T_1^1) = P'(n,\T_1) = \{S_{i-1,j+1},S_{i,j+1}\}.$$
This domino position is boxed, and the therefore the cases of our lemma that we have already proved apply.  Hence there is a sequence $\Sigma_2$ for $\Pi^*$ and a $(\T_1',\T_2') \in T_{\Sigma_2}^L((\T_1^1,\T_2^1))$ such that $P(n,\T_1') = P'$.  We set $\Sigma = \Sigma_1 \Sigma_2$.  When $P=\{S_{ij},S_{i+1,j}\}$ is vertical instead, we can conclude that $\phi_D(S_{ij}) = W$ and $\kappa_{j-1}(\T_1) = i+1$, and a transposed version of the argument above applies.

Finally, consider the case when $P'$ is unboxed.  First assume that $P'$ is horizontal, setting $P'=\{S_{ij},S_{i,j+1}\}$.  Arguing as above, we find $\phi_D(S_{ij}) = W$ and $\rho_{i-1}(\T_1) = j+1$.  Set $P''=\{S_{i-1,j+1},S_{i,j+1}\}$.  It is boxed and an extremal position in $\T_1$.  Using Propositions 4.\ref{proposition:456} and 4.\ref{proposition:454} if necessary, we can appeal to the cases of the lemma that we have already verified, letting $\Sigma_1$ be a sequence for $\Pi^*$ satisfying $(\T_1^1, \T_2^1) \in T_{\Sigma_1}^L((\T_1,\T_2))$ with $\T_2^1 = \T_2$ and $P(n,\T_1^1) = P''$. Now let $(\oT_1^1,\oT_2^1) = (\T_1^1,\T_2^1) -L$.  By Theorem 2.2.3, there is an open cycle $c \in OC(\oT_2^1)$ with $S_b(c) = S_{ij}$ and $S_f(c) = S_{i-1,j+1}$.  Again using Propositions 4.\ref{proposition:456} and 4.\ref{proposition:454} if necessary we can find a tableau $\oT$ such that $\sh(\oT) = \sh(\oT_1^1)$, $P(n-1,\oT) = \{S_{i-1,j},S_{ij}\}$, and $n_v(\oT) \equiv n_v(\oT_1^1) \pmod{4}$.  Then $\{n-1\}$ constitutes an extended cycle in $\oT$ relative to $\oT_2^1$.  Set $$(\oT_1^2,\oT_2^2) = \E((\oT_,\oT_2^1),\{n-1\},L),$$
so that $\sh(\oT_1^2)=(\sh(\oT)\setminus \{S_{ij}\})\cup \{S_{i-1,j+1}\}.$
Since $\Ss(\oT_2^2) = \oT_2^1$, we can use Theorem 4.\ref{theorem:samespecialthensequence} inductively to find a sequence $\Sigma_2$ for $\Pi^* \setminus \{\alpha_n\}$ such that $(\oT_1^2, \oT_2^2) \in T_{\Sigma_2}^L((\oT_1^1,\oT_2^1))$.  Let $(\T_1',\T_2') \in T_{\Sigma_2}^L((\T_1^1,\T_2^1))$ be such that $(\T_1',\T_2') - L = (\oT_1^2,\oT_2^2).$  Then $P(n,\T_1') = P'(n,\T_1^1) = P'$.  By setting $\Sigma = \Sigma_1 \Sigma_2$, we complete the proof of the lemma in this case.   When $P'$ is vertical, a transposed argument applies. This exhausts all of the possible cases, bringing the proof of the lemma to a merciful end.
\end{proof}

\section{}
\begin{definition} We extend Definition 3.4.1 (of the left generalized $\tau$-invariant), this time relative to type $D$, to $\mathcal{T}_D(n,n)$ in the obvious way, replacing the notion of sequence in \cite{garfinkle3} by that of Definition 4.5.1; we define the right generalized $\tau$-invariant similarly.  We denote this relation by $\sim_{GTLD}$ rather than $\sim_{GTL}$ for clarity.
\end{definition}

Then Remark 3.4.2 carries over to type $D$ with the same proof, as do 3.4.3--3.4.7, replacing the pair $\{\alpha,\beta\}$ wherever it occurs by $\{\alpha,X\}$, where $\alpha$ is one of the simple roots $\alpha_1',\alpha_2$, or $\alpha_4$ and $X=C$ or $D$.  One also replaces the conditions $1\in ec(2,\T_1;\T_2),c(2,\T_1)\in OC^*(\T_1)$ by $3\in ec((4,\T_1;\T_2),c(4,\T_1)\in OC^*(\T_1)$, respectively.   Thus for any tableau pairs $(\T_1,\T_2),(\T_1',\T_2')$ in $\mathcal{T}'(n,n)$ we have  $(\T_1,\T_2)\sim_{GTLD}(\T_1',\T_2')$ whenever $\Ss(\T_1) = \Ss(\T_2)$.  Corresponding to Lemma 3.4.8 we have 

\begin{lemma}\label{lemma:462}Let $(\T_1,\T_2)\in\mathcal{T}_D(M_1,M_2)$ with $M_1=\{1,\ldots,n\}$.  Let $1<l<n$ and let $\bar{\T}_1$ be the tableau obtained from $\T_1$ by removing the dominos labelled $l+1,\ldots,n$.  Let $\bar{P}$ be an extremal position in $\bar{\T}_1$.  Assume that there exists a tableau $\T$ having the same shape as $\bar{\T}_1$, the $l$-domino in position $\bar{P}$, and the same $n_v$ value as $n_v(\bar{\T}_1$ modulo 4.  Then there is a sequence $\Sigma$ for $\{\alpha_1',\alpha_2,\ldots,\alpha_l\}$ and a $(\T_3,\T_2)\in T_\Sigma^L((\T_1,\T_2))$ such that $P(l,\T_3)= \bar{P}$ and, for $l+1\le r\le n, P(r,\T_3) = P(r,\T_1)$.
\end{lemma}

\begin{proof} This is proved as in 3.4.8, using Lemma 4.5.9 instead of Theorem 3.2.2.
\end{proof}

Then 3.4.9--3.4.14 carry over to type $D$ with the same proofs.  Our main result extends Theorem 3.4.17 to type $D$.

\begin{theorem}\label{theorem:463}Let $(\T_1,\T_2),(\T_1',\T_2')$ both lie in  $\mathcal T_D'(M_1,M_2)$  or both lie in\linebreak $\mathcal T_D"(M_1,M_2)$,where $M_1=M_2 =\{1,\ldots,n\}$.  Then $(\T_1,\T_2)\sim_{GTLD}(\T_1',\T_2')$ if and only if $\Ss(\T_1) = \Ss(\T_2)$.
\end{theorem}

\begin{proof} This is proved in the same way as Theorem 3.4.17, using the analogues of Lemmas 3.4.15 and 3.4.16 for type $D$, except that (as noted above) Lemma 3.4.8 no longer holds for all tableau shapes.  A further complication is that specialness of a tableau shape is no longer preserved under transposition, so that the transposes of certain cases must be considered separately; fortunately the hypothesis of specialness is not used in many cases of the proof of Theorem 3.4.17.   For any tableau shape for which the conclusion of Lemma 3.4.8 fails in type $D$ (that is, a tableau shape obtained from one of the bad ones listed after Proposition 4.5.6 by adding a single domino), we therefore compute the generalized $\tau$-invariant explicitly and check directly that it does not coincide for the tableaux $\T_1$ and $\T_1'$.  In all cases the computation is easy.
\end{proof}

\section{}

We define the relations $\sim_{JLD}$ and $\sim_{JRD}$ on $W', \mathcal{S}'(n,n),\mathcal{S}''(n,n)$ and $\mathcal {T}_D(n,n)$ in exactly the same way for type $D$ as $\sim_{JR},\sim_{JL}$ were defined for types $B$ and $C$ in Definition 3.5.1; note that $\sim_{JR}, \sim_{JRD}$ are the {\sl left} and not the right cell relations of Joseph, as erroneously stated in that definition.

\begin{theorem}\label{theorem:471} Let $(\T_1,\T_2),(\T_1',\T_2')\in\mathcal{T}_D(n,n)$.  Then $(\T_1,\T_2)\sim_{JRD}(\T_1',\T_2')$ if and only if $\Ss(\T_1) = \Ss(\T_1')$.
\end{theorem}

\begin{proof} This follows at once from Theorem 4.5.2 (and its easy converse).
\end{proof}

Now 3.5.4--3.5.6 carry over immediately to type $D$ with the obvious changes in notation; note that the condition $w_1\sim_{JL} w_2$ in Proposition 3.5.4 (1) should read $w_1\sim_{JR} w_2$; similarly, the condition $w_1\sim_{GTR} w_2$ in (2) should read $w_1\sim_{GTL}(w_2)$.  Likewise the condition $w_1\sim_{GTR} w_2$ in Corollary 3.5.6 (1) should read $w_1\sim_{GTL} w_2$.  In the statements of these results for type $D$, we assume throughout that either $w_1,w_2\in W'$ or $w_1,w_2\in W"$.  Then Definition 3.5.7 carries over word for word to type $D$ and we have $I_\lambda(w_1) \sim_{GTD} I_\lambda(w_2)$ if and only if $w_1\sim_{GTL} w_2$.
We have

\begin{theorem}\label{theorem:472} Suppose all the simple factors of $\Delta_{\lambda}$ are of type $D$.  Let $I_1,I_2\in$Prim$_\lambda (U(\mathfrak g))$.  Then $I_1\sim_{GTD}I_2$ if and only if $I_1=I_2$.
\end{theorem}

We conclude with the definition of the operator $T_{\alpha_1'C}^L$ of Definition 4.4.10 on the set $\mathcal{T}_D^S(n)$ of tableaux of special shape relative to type $D$; we make the analogous definition for the other operators defined in Section 4.  If $C^3$ or $C^4$ is a subset of a tableau $\T$ of special shape, then the cycle $c(4,\T_1)$ is necessarily closed; move through this cycle and replace the subset $C^1$ (resp.\ $C_2$) in the resulting tableau by $B_1^1$ (resp.\ $B_1^2$) to obtain the single tableau $T_{\alpha_1'C}(\T)$.  Otherwise $C^1$ or $C^2$ is a subset of $\T$; replace it by $B^1$ or $B^2$ to obtain a tableau $\T'$.  if $c(4,\T')$ is open, take $T_{\alpha_1'C}(\T)$ to be the single tableau $\T'$.  Finally, if $c(4,\T)$ is closed, $T_{\alpha_1'C}(\T)$ has two values, the first one being $\T'$ and the second one obtained from it by moving through $c(4,\T')$.

Then the classification theorem reads

\begin{theorem}\label{theorem:473} Suppose that $\mathfrak g$ is of type $D_n$ and $\lambda\in\mathfrak h$ is integral, regular, and anti-dominant.  Then the map $cl$ from Prim$_\lambda (U(\mathfrak g))$ to $\mathcal T_D^S(n)$ sending $I_\lambda(w)$ to $\Ss(\LL(\delta(w)))$ is a bijection; it is also the unique map from its domain to its range preserving $\tau$-invariants and commuting with the operators $T_{\alpha\beta}$ and the operators of Definition 4.4.3.
\end{theorem}

\end{document}